\documentclass[11pt,oneside,english]{amsart}
\usepackage[T1]{fontenc}
\usepackage[latin9]{inputenc}
\usepackage{babel}
\usepackage{textcomp}
\usepackage{amsthm}
\usepackage{amssymb}
\usepackage{setspace}
\usepackage{esint}
\usepackage{graphicx}
\usepackage{enumerate}

\usepackage[unicode=true,pdfusetitle,
 bookmarks=true,bookmarksnumbered=false,bookmarksopen=false,
 breaklinks=false,pdfborder={0 0 1},backref=false,colorlinks=false]
 {hyperref}

\makeatletter
\numberwithin{equation}{section}
\numberwithin{figure}{section}
\theoremstyle{plain}
\newtheorem{thm}{\protect\theoremname}[section]
  \theoremstyle{plain}
  \newtheorem{prop}[thm]{\protect\propositionname}
  \theoremstyle{remark}
  \newtheorem*{rem*}{\protect\remarkname}
  \theoremstyle{plain}
  \newtheorem{conjecture}[thm]{\protect\conjecturename}
  \theoremstyle{plain}
  \newtheorem{lem}[thm]{\protect\lemmaname}
  \theoremstyle{plain}
  \newtheorem{cor}[thm]{\protect\corollaryname}

\usepackage{ifpdf} 
\ifpdf 

 \IfFileExists{lmodern.sty}{\usepackage{lmodern}}{}

\fi 

\makeatother

  \providecommand{\conjecturename}{Conjecture}
  \providecommand{\corollaryname}{Corollary}
  \providecommand{\lemmaname}{Lemma}
  \providecommand{\propositionname}{Proposition}
  \providecommand{\remarkname}{Remark}
\providecommand{\theoremname}{Theorem}

\begin{document}

\title{Pair correlation for quadratic polynomials mod 1}

\author{Jens Marklof and Nadav Yesha}
\begin{abstract}
It is an open question whether the fractional parts of non-linear polynomials at integers have the same fine-scale statistics as a Poisson point process. Most results towards an affirmative answer have so far been  restricted to almost sure convergence in the space of polynomials of a given degree. We will here provide explicit Diophantine conditions on the coefficients of polynomials of degree 2, under which the convergence of an averaged pair correlation density can be established. The limit is consistent with the Poisson distribution. Since quadratic polynomials at integers represent the energy levels of a class of integrable quantum systems, our findings provide further evidence for the Berry-Tabor conjecture in the theory of quantum chaos.
\end{abstract}

\date{\today}

\address{Jens Marklof, School of Mathematics, University of Bristol, Bristol
BS8 1TW, U.K.}

\email{J.Marklof@bristol.ac.uk}

\address{Nadav Yesha, School of Mathematics, University of Bristol, Bristol
BS8 1TW, U.K.}

\curraddr{Department of Mathematics, King's College London, Strand, London WC2R 2LS, U.K.} \email{nadav.yesha@kcl.ac.uk}

\maketitle

\section{Introduction}

Let $\varphi(z)=\alpha_q z^q + \alpha_{q-1} z^{q-1} +\ldots + \alpha_0$ be a polynomial of degree $q$. In his ground-breaking 1916 paper \cite{Weyl}, H.~Weyl proved that, if one of the coefficients $\alpha_q,\ldots,\alpha_1$ is irrational, then the sequence $\varphi(1),\varphi(2),\varphi(3),\ldots$ is uniformly distributed mod 1. That is, for any interval $A\subset\mathbb{R}$ of length $|A|\leq 1$ we have
\begin{equation}\label{eq:mod1}
\lim_{N\to\infty} \frac{\#\left\{ j\leq N \mid \varphi(j) \in A +\mathbb{Z}  \right\}}{N} = |A| .
\end{equation}
A century after Weyl's work, it is remarkable how little is known on the fine-scale statistics of $\varphi(j)$ mod 1 in the case of polynomials $\varphi$ of degree $q\geq 2$. (The linear case $q=1$ is singular and related to the famous three gap theorem; cf.~\cite{Slater,Bleher,Greenman,PandeyBohigasGiannoni}).
A popular example of such a statistics is the {\em gap distribution}:
for given $N$, order the values 
$\varphi(1), \ldots,\varphi(N)$ in $[0,1)$ mod 1, and label them as $$0\leq \theta_{1,N}\leq \theta_{2,N}\leq\ldots\leq \theta_{N,N}<1.$$  
As we are working mod 1, it is convenient to set $\theta_{0,N}:=\theta_{N,N}-1$.
The gap distribution of $\{\varphi(j)\}_{j=1}^N$ is then given by the probability measure $P_N(\,\cdot\,;\varphi)$ defined, for any interval $A\subset [0,\infty)$, by
\begin{equation}
P_N(A;\varphi) = \frac{\#\left\{ j\leq N \mid \theta_{j,N}-\theta_{j-1,N} \in N^{-1} A\right\}}{N}.
\end{equation}

Rudnick and Sarnak \cite{RudnickSarnak} conjecture the following.

\begin{conjecture}\label{gapcon}
Let $q\ge 2$. There is a set $Q\subset\mathbb{R}$ of full Lebesgue measure, such that for $\alpha_q\in Q$ and any interval $A\subset [0,\infty)$,
\begin{equation}\label{eq:congap}
\lim_{N\to\infty} P_N(A;\varphi) = \int_{A}e^{-s}\,\mbox{d}s .
\end{equation}
\end{conjecture}

That is, the gap distribution of a degree $q$ polynomial with typical leading coefficient convergences to the distribution of waiting times of a Poisson process with intensity one. Sinai \cite{Sinai} pointed out that, for quadratic polynomials (with random coefficients), high moments of the fine-scale distribution do not converge; this is of course not in contradiction with Conjecture \ref{gapcon}, which only concerns weak convergence. Pellegrinotti on the other hand pointed out that the first four moments do converge to the Poisson distribution \cite{Pellegrinotti}.

The validity of Conjecture \ref{gapcon} is completely open. It is known that the conjecture cannot hold for every irrational $\alpha_q$. The general belief, however, is that $Q$ includes {\em every} irrational number of Diophantine type $2+\epsilon$, for any $\epsilon>0$ (see \eqref{def:dio} below); this includes for instance all algebraic numbers \cite{RudnickSarnak}. The only result to-date on the gap distribution, due to Rudnick, Sarnak and Zaharescu \cite{RudnickSarnakZaharescu}, holds for quadratic polynomials, and establishes the convergence in \eqref{eq:congap} along subsequences for leading coefficients $\alpha_2$ that are well approximable by rationals. Unfortunately, for these $\alpha_2$ we cannot expect convergence along the full sequence, as they do not have the required Diophantine type.



A more accessible fine-scale statistics is the pair correlation measure $R_{2,N}(\,\cdot\,;\varphi)$ which, for any bounded interval $A\subset\mathbb{R}$, is defined by
\begin{equation}\label{eq:R2Def}
R_{2,N}(A;\varphi) 
= \frac{\#\left\{ (i,j) \mid i\neq j\leq N, \; \varphi(i)-\varphi(j) \in N^{-1} A + \mathbb{Z} \right\}}{N} .
\end{equation}
Rudnick and Sarnak proved that the pair correlation measure of polynomials $P(z)=\alpha z^q$ converges for almost every $\alpha$ to Lebesgue measure, the pair correlation measure of a Poisson point process.

\begin{thm}[Rudnick-Sarnak \cite{RudnickSarnak}]\label{thm:pair}
Let $q\ge 2$. There is a set $Q\subset\mathbb{R}$ of full Lebesgue measure, such that for $\varphi(z)=\alpha z^q$, $\alpha\in Q$, and any bounded interval $A\subset\mathbb{R}$,
\begin{equation}\label{eq:conpair}
\lim_{N\to\infty} R_{2,N}(A;\varphi) = |A| .
\end{equation}
\end{thm}

It is currently not known whether any Diophantine condition on $\alpha$ will guarantee the convergence in \eqref{eq:conpair}. In the quadratic case $q=2$, Heath-Brown \cite{HeathBrown} gave an explicit construction of $\alpha$'s such that \eqref{eq:conpair} holds. The question whether \eqref{eq:conpair} holds for Diophantine $\alpha$ (which holds, e.g., for $\alpha=\sqrt2$), however, remains open to this day. See Truelsen \cite{Truelsen} for a conditional result, and Marklof-Str\"ombergsson \cite{MarklofStrombergsson} for a derivation of \eqref{eq:conpair} from a geometric equidistribution result. Boca and Zaharescu \cite{BocaZaharescu} generalized Theorem \ref{thm:pair} to $\varphi(z)=\alpha f(z)$, where $f$ is any polynomial of degree $q$ with integer coefficients. 

Zelditch proved an analogue of Theorem \ref{thm:pair} for the sequence of polynomials
\begin{equation}
\varphi_N(z)=\alpha N f\left(\frac{z}{N}\right) + \beta z,
\end{equation}
where $f$ is a fixed polynomial satisfying $f''\neq 0$ on $[-1,1]$.
Surprisingly, the pair correlation problem for $\varphi_N$ seems harder than the case of fixed polynomials $\varphi$, and requires an additional Ces\`aro average:

\begin{thm}[Zelditch {\cite[add.]{Zelditch}}]\label{thm:pair2}
There is a set $Q\subset\mathbb{R}^2$ of full Lebesgue measure, such that for $(\alpha,\beta)\in Q$, and any bounded interval $A\subset\mathbb{R}$,
\begin{equation}\label{eq:conpairZelditch}
\lim_{M\to\infty} \frac1M \sum_{N=1}^M R_{2,N}(A;\varphi_N) = |A| .
\end{equation}
\end{thm}

Slightly weaker results hold when $f$ is only assumed to be a smooth function \cite{Zelditch}. The motivation for the particular choice of $\varphi_N$ is that the values $\{\varphi_N(j)\}_{j=1}^N$ represent the eigenphases of quantized maps \cite{Zelditch}; the understanding of their statistical distribution is one of the central challenges in quantum chaos \cite{BerryTabor,MarklofECM}. Zelditch and Zworski have found analogous results for scattering phase shifts \cite{ZelditchZworski}.

The aim of the present paper is to improve Theorem \ref{thm:pair2} for the specific example
\begin{equation}\label{ours}
\varphi_N(z)=\frac{(z-\alpha)^2}{2N} .
\end{equation}
The corresponding triangular array $\{ \varphi_N(j)\}_{1\le j\le N,\, 1\le N<\infty}$, is uniformly distributed mod 1: For any bounded interval $A\subset\mathbb{R}$,
\begin{equation}\label{eq:mod1Phi}
\lim_{N\to\infty} \frac{\#\left\{ j\leq N \mid \varphi_N(j) \in A +\mathbb{Z} \right\}}{N} = |A|.
\end{equation}
Appendix \ref{sec:Uniform Distribution of} comprises a precise bound on the discrepancy.

We will establish convergence of the pair correlation measure under an explicit Diophantine condition on $\alpha$, rather than convergence almost everywhere as in Theorem \ref{thm:pair2}. Furthermore, we will shorten the Ces\`aro average and provide a power-saving in the rate of convergence.

\begin{figure}
	\includegraphics[scale=0.28]{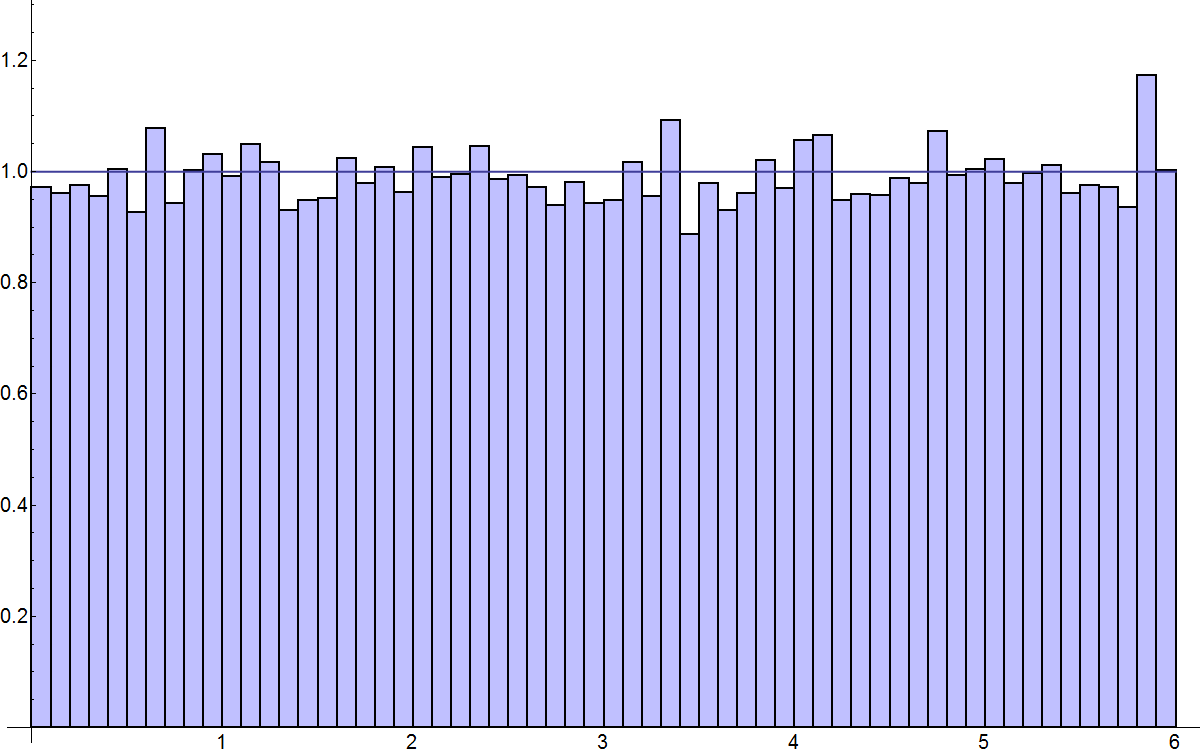}
	\caption{Pair correlation of $\left\{\frac{(n-\sqrt{2})^{2}}{2N}\right\}_{n=1}^N$ mod $ 1 $, with $ N = 5000 $.}\label{fig:pair}
\end{figure}

\begin{figure}
	\includegraphics[scale=0.28]{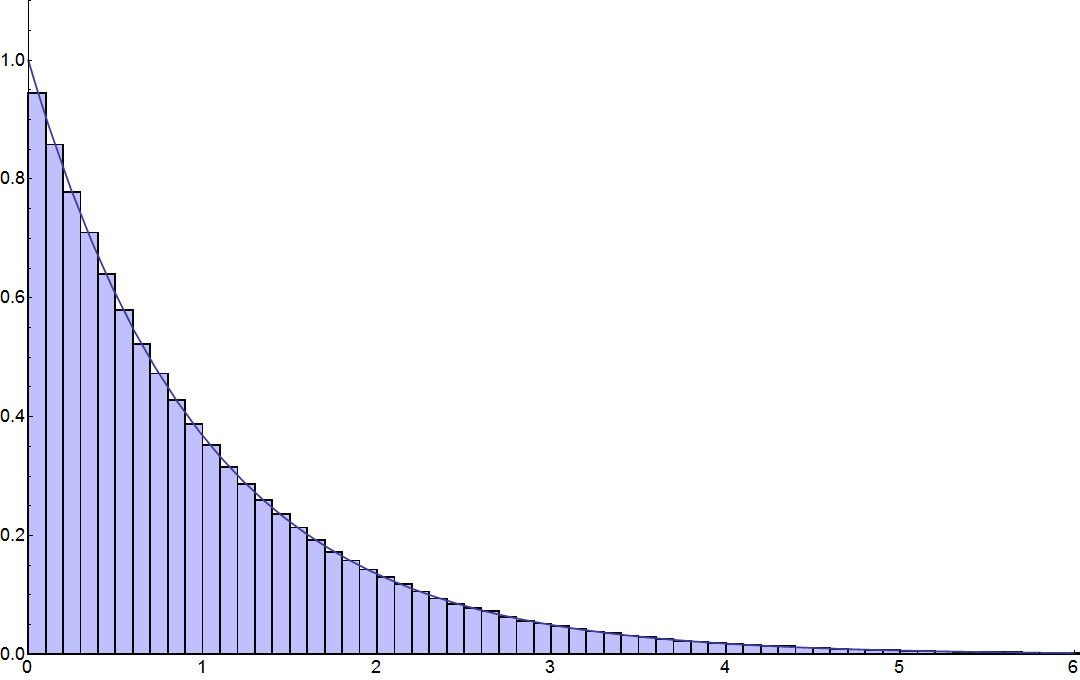}
	\caption{Gap distribution of $\left\{\frac{(n-\sqrt{2})^{2}}{2N}\right\}_{n=1}^N$ mod $ 1 $, with $ N = 10^6 $.}\label{fig:gap}
\end{figure}

An irrational $\alpha\in\mathbb{R}$ is said to be {\em Diophantine}, if there exists $\kappa,C>0$ such that
\begin{equation}\label{def:dio}
\left|\alpha-\frac{p}{q}\right|>\frac{C}{q^{\kappa}}
\end{equation}
for all $p\in\mathbb{Z},\,q\in\mathbb{N}$. The constant $\kappa$ is called the {\em Diophantine type} of $\alpha$; by Dirichlet's theorem,
$\kappa\ge2$. Every quadratic surd, and more generally every $\alpha$ with bounded continued fraction expansion, is of Diophantine type $\kappa=2$.

\begin{thm}
\label{thm:MainThm} Choose $\varphi_N$ as in \eqref{ours}, with $\alpha$ Diophantine of type $\kappa$,
and fix $$\eta\in\left( \max\left(\frac{17}{18},1-\frac{2}{9\kappa}\right),1\right].$$
There exists $s=s(\kappa,\eta)>0$ such that for any bounded interval $A\subset\mathbb{R}$
\begin{equation} \label{eq:pair3}
\frac{1}{M^{\eta}}\sum_{M\le N\le M+M^{\eta}}R_{2,N}(A;\varphi_N)=|A| +O\left(M^{-s}\right)
\end{equation}
as $M\to\infty$ (the implied constant in the remainder depends on $\alpha,\eta,A$).
\end{thm}

Theorem \ref{thm:MainThm} extends to more general pair correlation measures, see Appendix \ref{sec:MoreGeneral} for a discussion. Numerical experiments (see Fig.~\ref{fig:pair}) seem to suggest that the average over $N$ in \eqref{eq:pair3} might in fact not be necessary in \eqref{eq:pair3}, and that furthermore the gap distribution is exponential (Fig.~\ref{fig:gap}). 

The proof of this statement reduces to a natural equidistribution theorem on a three-dimensional Heisenberg manifold, which we show derives from Str\"ombergsson's recent quantitative Ratner equidistribution result on the space of affine lattices. Assuming that a more subtle equidistribution result holds (Conjecture \ref{conj:EquiConj}), we can remove the average over $N$ in \eqref{eq:pair3} and obtain the limit $R_{2,N}(A;\varphi_N)\to |A|$; see Proposition \ref{prop:ConjecturedProp}.

The fact that $ \alpha $ is irrational and badly approximable is a crucial assumption. In the case $ \alpha=0 $, the gap distribution reduces to the distribution of spacing between quadratic residues mod $ 2N $, which was proved by Kurlberg and Rudnick \cite{KurlbergRudnick} to be Poisson along odd, square-free and highly composite $ N $'s.

If we do not insist on a rate of convergence, and also allow for a large average,  Theorem \ref{thm:MainThm} holds in fact for the more general polynomials
\begin{equation}\label{ours2}
\varphi_N(z)=\beta\; \frac{(z-\alpha)^2}{2N} ,
\end{equation}
where $\beta\neq 0$ is Diophantine and $\alpha\in\mathbb{R}$ arbitrary, or $\beta\neq 0$ arbitrary and $\alpha\in\mathbb{R}$ Diophantine. We will show in Appendix \ref{sec:LongAverages} that under these assumptions
\begin{equation}
\lim_{N\to\infty} \frac{1}{M}\sum_{M\le N\le 2M}R_{2,N}(A;\varphi_N)=|A| .
\end{equation}
This statement is an immediate corollary of the quantitative Oppenheim conjecture for quadratic forms of signature $(2,2)$, first proved by Eskin, Margulis and Mozes \cite{EskinMargulisMozes} for homogeneous forms,
and generalized to inhomogeneous forms by Margulis and Mohammadi \cite{MargulisMohammadi}.

We conclude this introduction by briefly describing the relevance of our results to the theory of quantum chaos. The values $\{\varphi_N(j)\}_{j=1}^N$ for $\varphi_N$ as in \eqref{ours}, or more generally \eqref{ours2}, represent the eigenphases of a particularly simple quantum map studied by De Bièvre, Degli Esposti and Giachetti \cite[\S4B]{DeBievreEspostiGiachetti}: a quantized shear of a toroidal phase space with quasi-periodic boundary conditions, where $\beta$ quantifies the shear-strength and $\alpha$ is related to the quasi-periodic boundary condition. Theorem \ref{thm:MainThm} proves that the average two-point spectral statistics of this quantum map are Poisson. 

The spacings in the sequence $\{\varphi_N(j)\}_j$ furthermore represent the quantum energy level spacings of a class of integrable Hamiltonian systems: a particle on a ring with quasi-periodic boundary condition (representing a magnetic flux line through its center), coupled to a harmonic oscillator. The corresponding energy levels are (in suitable units) $E_{j,k}(\hbar)=\frac \beta 2 \hbar^2 (j-\alpha)^2 +\hbar (k-\frac12)$, where $ j\in\mathbb Z$, $k\in \mathbb N$ and $\beta>0$. 
It is a short exercise to show that, after rescaling by $\hbar$, the energy levels $E_{j,k}(\hbar)\in [\overline E-\frac{\hbar}{2},\overline E-\frac{\hbar}{2})$ have the same spacing statistics as $\varphi_N(j)\mod 1$ for $\hbar=N^{-1}$ and $|j-\alpha|< \left(\frac{2\overline{E}} {\beta} \right)^{1/2} N$. The index range is thus not $j\leq N$ as in \eqref{eq:R2Def}; cf.~Appendices \ref{sec:MoreGeneral} and \ref{sec:LongAverages} for the relevant generalizations.

Berry and Tabor \cite{BerryTabor} conjectured that typical integrable quantum systems should have Poisson statistics in the semiclassical limit $\hbar\to 0$, and the results presented in this study may therefore be viewed as further evidence for the truth of the Berry-Tabor conjecture. Other instances in which the conjecture could be rigorously established are reviewed in \cite{MarklofECM}.

\subsection*{Acknowledgements}
We thank Zeév Rudnick for helpful comments. The research leading to these results has received funding from the European Research Council under the European Union's Seventh Framework Programme (FP/2007-2013) / ERC Grant Agreement n. 291147.

\section{Outline of the proof: Smoothing and geometric equidistribution}
 
We will first prove a smooth version of Theorem \ref{thm:MainThm}. Let $ \mathcal{S}(\mathbb{R}) $ be the Schwartz space of rapidly decreasing smooth functions on $ \mathbb{R} $, equipped with the norms \[
\left\Vert f\right\Vert _{C^{k}}=\sum_{0\le l\le k}\left\Vert f^{\left(l\right)}\right\Vert _{\infty},
\] and denote \[ \hat{f}(\xi)= \int_{-\infty}^\infty{f(x)e(-x \xi)}\,\mbox{d}x \] the Fourier transform of $ f\in\mathcal{S}(\mathbb{R})$, where we use the standard notation $e\left(x\right)=\exp\left(2\pi ix\right)$. For $ f,h\in \mathcal{S}(\mathbb{R})  $, define the smooth pair correlation function
\begin{equation} \label{eq:smoothR2}
R_{2,N}\left(f,h;\varphi\right)=\frac{1}{N}\sum_{m\in\mathbb{Z}}\sum_{i,j\in\mathbb{Z}}h\left(\frac{i-\alpha}{N}\right)h\left(\frac{j-\alpha}{N}\right)f\left(N\left(\varphi(i)-\varphi(j)+m\right)\right);
\end{equation}
for technical reasons, introducing the shifts by $ \alpha $ inside $h$ in (\ref{eq:smoothR2}) gives a more convenient approximation to the sharp cutoff in (\ref{eq:R2Def}).

Throughout the remainder of this paper we will use $\varphi(z)=\varphi_N(z)=\beta \, \frac{\left(z-\alpha\right)^{2}}{2N}$ with $\beta=1$ as in \eqref{ours}, except for Appendix \ref{sec:LongAverages} where we consider general values of $\beta$.


We prove the following smoothed variant of Theorem \ref{thm:MainThm}:
\begin{thm}
\label{thm:MainThmSmooth}Fix a Diophantine $\alpha$ of type $\kappa$,
and let $\max\left(\frac{17}{18},1-\frac{2}{9\kappa}\right)<\eta \le 1$ be fixed.
There exist $s=s(\kappa,\eta)>0,n=n(\kappa,\eta)\in\mathbb N$ such that for any $\nu\in C^{\infty}\left(\mathbb{R}\right)$ supported on $ \left[-L,L\right] $ $\left(L\ge1\right)$ and
 $\psi,h\in C^{\infty}\left(\mathbb{R}\right)$ real valued and supported on $ \left[-1/2,3/2\right] $,
\begin{alignat*}{1}
\frac{1}{M^{\eta}}\sum_{N\in\mathbb{Z}}\psi\left(\frac{N-M}{M^\eta}\right)R_{2,N}\left(\hat{\nu},h;\varphi_N\right) & = \hat{\psi}(0)\left(\nu\left(0\right)\left|\hat{h}(0)\right|^{2}+\hat{\nu}\left(0\right)\left\Vert h\right\Vert _{2}^{2} \right) \\
 & +O\left(L\left\Vert h\right\Vert _{C^{n}}^2\left\Vert \nu\right\Vert _{C^{2}} \left\Vert \psi \right\Vert _{C^{8}} M^{-s}\right)
\end{alignat*}
as $M\to\infty$ (the implied constant in the remainder depends on $\alpha,\eta$).
\end{thm}

(The specific choice of the interval $[-1/2,3/2]$ in
the above is simply for convenience. The proof generalizes to general
bounded intervals.)

Define
\[
Q_{N}\left(\nu,h\right)=\frac{1}{N}\sum_{k\ne0}\nu\left(\frac{k}{N}\right)\left|\frac{1}{\sqrt{N}}\sum_{n\in\mathbb{Z}}h\left(\frac{n-\alpha}{N}\right)e\left(k\frac{\left(n-\alpha\right)^{2}}{2N}\right)\right|^{2}.
\]

Theorem \ref{thm:MainThmSmooth} will follow from the following proposition,
which is key in this paper:
\begin{prop}
\label{prop:KeyProp}Fix a Diophantine $\alpha$ of type $\kappa$,
and let $\max\left(\frac{17}{18},1-\frac{2}{9\kappa}\right)<\eta \le 1$ be fixed.
There exist $s=s(\kappa,\eta)>0,n=n(\kappa,\eta)\in\mathbb N$ such that for any $\nu\in C^{\infty}\left(\mathbb{R}\right)$ supported on $ \left[-L,L\right] $ $\left(L\ge1\right)$ and
$\psi,h\in C^{\infty}\left(\mathbb{R}\right)$ real valued and supported on $ \left[-1/2,3/2\right] $,

\begin{alignat}{1}\label{eq:PropEq}
\frac{1}{M^{\eta}}\sum_{N\in\mathbb{Z}}\psi\left(\frac{N-M}{M^\eta}\right)Q_{N}\left(\nu,h\right)&=\hat{\psi}(0)\hat{\nu}\left(0\right)\left\Vert h\right\Vert _{2}^{2} \\
&+O\left(L\left\Vert h\right\Vert _{C^{n}}^2\left\Vert \nu\right\Vert _{C^{2}} \left\Vert \psi \right\Vert _{C^{8}} M^{-s}\right). \nonumber
\end{alignat}
as $M\to\infty.$
\end{prop}

The strategy of the proof of Proposition \ref{prop:KeyProp} which will be given in Section \ref{sec:Prop2.2Proof}, is first
to interpret the l.h.s. of (\ref{eq:PropEq}) as a smooth sum of
the absolute square of the Jacobi theta sum. Then we establish Proposition \ref{prop:KeyProp} from a geometric equidistribution result, which in turn follows from an effective Ratner equidistribution result due to Strömbergsson \cite{Strombergsson}.

In order to give the exact formulation of our equidistribution result, let $\mu_{H}$ be the Haar measure on the Heisenberg group \[\mathbb{H}\left(\mathbb{R}\right)= \left\{\begin{pmatrix}
1 & u & \xi_1\\
0 & 1 & \xi_2 \\
0 & 0 & 1
\end{pmatrix}: u\in \mathbb{R},\xi=(\xi_1,\xi_2) \in \mathbb{R}^2\right\},\]
normalized such that it induces a probability measure on $\mathbb{H}\left(\mathbb{Z}\right)\backslash\mathbb{H}\left(\mathbb{R}\right)$; the latter is also denoted by $\mu_{H}$. 

Denote the differential operators 
\[
D_{1}=v^{2}\frac{\partial}{\partial u},D_{2}=v\frac{\partial}{\partial v},D_{3}=v \frac{\partial}{\partial\xi_{1}},D_{4}=\frac{u}{v}\frac{\partial}{\partial\xi_{1}}+\frac{1}{v}\frac{\partial}{\partial\xi_{2}}.
\]
For a compactly supported $f=f(v,u,\xi) \in C^{k}\left((0,\infty) \times \mathbb{H}\left(\mathbb{Z}\right)\backslash\mathbb{H}\left(\mathbb{R}\right)\right),$
 define the norm
\[
\left\Vert f\right\Vert _{C^{k}}=\sum_{\operatorname{ord}D\le k}\left\Vert Df\right\Vert _{\infty}
\]
summing over all monomials in $D_{i}$ $\left(1\le i\le4\right)$ of
degree $\le k$.

Recall that a vector $\eta\in\mathbb{R}^{2}$ is called Diophantine
of type $\kappa$, if there exists $C>0$ such that 
\[
\left\Vert \eta-\frac{m}{q}\right\Vert >\frac{C}{q^{\kappa}}
\]
for all $m\in\mathbb{Z}^{2},\,q\in\mathbb{N}$. By Dirichlet's theorem,
$\kappa\ge3/2$.

\begin{prop}\label{prop:equi_prop}
Fix a Diophantine vector $ \eta $ of type $\kappa$, $\delta>0$. Then
for any $f\in C^{8}\left((0,\infty) \times \mathbb{H}\left(\mathbb{Z}\right) \backslash\mathbb{H}\left(\mathbb{R}\right)\right)
$ compactly supported on $ [1/4, \infty) \times \mathbb{H}\left(\mathbb{Z}\right) \backslash\mathbb{H}\left(\mathbb{R}\right) $ and $\nu\in C^{\infty}\left(\mathbb{R}\right)$ supported on $ \left[-L,L\right] $ $\left(L\ge1\right)$,
we have
\begin{alignat*}{1}
	& \frac{1}{M^{2}}\sum_{\begin{subarray}{c}
			c,d\\
			\left(c,d\right)=1
		\end{subarray}}\nu\left(\frac{d}{c}\right)f\left(\frac{M}{c}, \frac{a}{c},\begin{pmatrix}
		a & b\\
		c & d
	\end{pmatrix}\eta \right)\\
	& =\frac{6}{\pi^{2}} \hat{\nu}(0)\int_0^\infty \int_{ \mathbb{H}\left(\mathbb{Z}\right) \backslash\mathbb{H}\left(\mathbb{R}\right)} f\left(v,u,\xi\right) v^{-3}\,\mbox{d}\mu_{H}\mbox{d}v\\
	&+O\left(L\left\Vert \nu\right\Vert _{C^{2}}\left\Vert f\right\Vert _{C^{8}}M^{-\min\left(1/2,2/\kappa\right)+\delta}\right)
\end{alignat*}
	as $M\to\infty$, where $\left(a,b\right)$ are any integers that
	solve the equation $ad-bc=1$.
\end{prop}

To give the strategy behind the proof of Proposition \ref{prop:equi_prop}, let $G$ be the semi-direct product Lie group $G=\mbox{SL}\left(2,\mathbb{R}\right)\ltimes\mathbb{R}^{2}$
with the multiplication law
\[
\left(g,\xi\right)\left(g',\xi'\right)=\left(gg',g\xi'+\xi\right),
\]
where the elements of $\mathbb{R}^{2}$ are viewed as column vectors, and let $\Gamma$ be the subgroup $\Gamma=\mbox{SL}\left(2,\mathbb{Z}\right)\ltimes\mathbb{Z}^{2}$. Denote \[
n_{+}\left(x\right)=\left(\begin{pmatrix}
1 & x\\
0 & 1
\end{pmatrix},0\right),a\left(y\right)=\left(\begin{pmatrix}
y & 0\\
0 & y^{-1}
\end{pmatrix},0\right).
\]We use the fact that, for every $0<v_0< v_1$, the set
\[ \Gamma \backslash \Gamma \left\{ (1,\xi)n_+(u)a(v): v_0<v<v_1, (u,\xi)\in \mathbb{H}\left( \mathbb{R} \right) \right\} \]is an embedding of $(v_0,v_1) \times \mathbb{H}\left(\mathbb{Z}\right)\backslash \mathbb{H}\left(\mathbb{R}\right)$ as a submanifold of $ \Gamma \backslash G $. By thickening this submanifold, we can use Strömbergsson's result to obtain equidistribution of our points.

As for the pointwise limit of the pair correlation measure
$R_{2,N}$, we show in Appendix \ref{sec:Equidistribution-on-Heisenberg}
that it also converges to the Lebesgue measure (and is thus consistent with the Poisson distribution), under the following equidistribution
conjecture on the Heisenberg group $\mathbb{H}\left(\mathbb{Z}\right)\backslash\mathbb{H}\left(\mathbb{R}\right)$, which is a pointwise version of Proposition \ref{prop:equi_prop} in the special case $ \eta = (1/2, \alpha) $:

For $f\in C^{k}\left(\mathbb{H}\left(\mathbb{Z}\right)\backslash\mathbb{H}\left(\mathbb{R}\right)\right),$
define the norm
\[
\left\Vert f\right\Vert _{C^{k}}=\sum_{\operatorname{ord}D\le k}\left\Vert Df\right\Vert _{\infty}
\]
summing over all monomials in the standard basis elements of the Lie
algebra of $\mathbb{H}\left(\mathbb{R}\right)$ (which correspond
to left-invariant differential operators on $\mathbb{H}\left(\mathbb{R}\right)$)
of degree $\le k$. 
\begin{conjecture}
\label{conj:EquiConj}There exist $k,l\in\mathbb{N},\kappa_{0}\ge2$,
such that for any fixed Diophantine  $\alpha$ of type $\kappa\le\kappa_{0},$ there exists $s=s(\kappa)>0$ such that for any $f\in C^{k}\left(\mathbb{H}\left(\mathbb{Z}\right)\backslash\mathbb{H}\left(\mathbb{R}\right)\right)$ and
$\nu\in C^{\infty}\left(\mathbb{R}\right)$ supported on $ \left[-L,L\right] $ $\left(L\ge1\right)$, we have

\begin{alignat*}{1}
\frac{1}{\phi\left(c\right)}\sum_{\begin{subarray}{c}
		d\in\mathbb{Z}\\
		\left(c,d\right)=1
\end{subarray}}\nu\left(\frac{d}{c}\right)f\left(\frac{a}{c},b\alpha+\frac{a}{2},d\alpha+\frac{c}{2}\right) & =\hat{\nu}\left(0\right)\int_{ \mathbb{H}\left(\mathbb{Z}\right) \backslash\mathbb{H}\left(\mathbb{R}\right)} f\,d\mu_{H}\\
& +O\left(L\left\Vert \nu\right\Vert _{C^{l}}\left\Vert f\right\Vert _{C^{k}}c^{-s}\right)
\end{alignat*}
as $c\to\infty$, where $\left(a,b\right)$ are any integers that
solve the equation $ad-bc=1$ and $\phi$ is the Euler totient function.
\end{conjecture}
We show in Appendix \ref{sec:Equidistribution-on-Heisenberg}:

\begin{prop}
	\label{prop:ConjecturedProp}Assuming Conjecture \ref{conj:EquiConj},
	there exists $ \kappa_0 \ge2 $, such that for any fixed Diophantine  $\alpha$ of type $ \kappa \le \kappa_0 $, there exists $s=s(\kappa)>0$ such that for any bounded interval $A\subset\mathbb{R}$
	\[
	R_{2,N}\left(A;\varphi_N\right)=|A|+O\left(N^{-s}\right)
	\]
	as $N\to\infty.$
\end{prop}

\section{Background}
\subsection{Effective Ratner equidistribution theorem}
Recall that $G$ is the semi-direct product Lie group $G=\mbox{SL}\left(2,\mathbb{R}\right)\ltimes\mathbb{R}^{2}$
with the multiplication law
\[
\left(g,\xi\right)\left(g',\xi'\right)=\left(gg',g\xi'+\xi\right),
\]
and that $\Gamma$ is the subgroup $\Gamma=\mbox{SL}\left(2,\mathbb{Z}\right)\ltimes\mathbb{Z}^{2}$.
Let $\mu$ be the Haar measure on $G$, normalized such that it
induces a probability measure on $\Gamma\backslash G$ (also denoted
by $\mu$).

Denote $G'=\mbox{SL}\left(2,\mathbb{R}\right)$ and $\Gamma'=\mbox{SL}\left(2,\mathbb{Z}\right)$,
which are subgroups of $G$ and $\Gamma$ under the embedding $g\mapsto\left(g,0\right)$.
We will also use the following notations:
\[
n_{+}\left(x\right)=\begin{pmatrix}
1 & x\\
0 & 1
\end{pmatrix},\,n_{-}\left(x\right)=\begin{pmatrix}
1 & 0\\
x & 1
\end{pmatrix},\,a\left(y\right)=\begin{pmatrix}
y & 0\\
0 & y^{-1}
\end{pmatrix}
\]
as well as 
\[
\Gamma'_{\infty}=\left\{ n_{+}\left(m\right):\,m\in\mathbb{Z}\right\} ,\,r\left(\phi\right)=\begin{pmatrix}
\cos\phi & -\sin\phi\\
\sin\phi & \cos\phi
\end{pmatrix}.
\]

A key ingredient in our proof will be an effective Ratner equidistribution
theorem due to Strömbergsson \cite{Strombergsson}.
%
%
Let $\mathfrak{g}$ be the Lie algebra of $G$, which may be identified
with the space $\mathfrak{sl}\left(2,\mathbb{R}\right)\oplus\mathbb{R}^{2}$
with the Lie bracket $\left[\left(X,v\right),\left(Y,w\right)\right]=\left(XY-YX,Xw-Yv\right).$
Fix the following basis of $\mathfrak{g}$:
\begin{alignat*}{1}
X_{1} & =\left(\begin{pmatrix}
0 & 1\\
0 & 0
\end{pmatrix},0\right),\hspace{1em}X_{2}=\left(\begin{pmatrix}
0 & 0\\
1 & 0
\end{pmatrix},0\right),\hspace{1em}X_{3}=\left(\begin{pmatrix}
1 & 0\\
0 & -1
\end{pmatrix},0\right)\\
X_{4} & =\left(0,\left(1,0\right)\right),\hspace{1em}X_{5}=\left(0,\left(0,1\right)\right).
\end{alignat*}

We define the space $C_{b}^{k}\left(\Gamma\backslash G\right)$ of
 $k$ times continuously differentiable functions on $\Gamma\backslash G$
 such that for any left-invariant
 differential operator $D$ on $G$ of order $\le k$, we have $\left\Vert Df\right\Vert _{\infty}<\infty$.
For $f\in C_{b}^{k}\left(\Gamma\backslash G\right),$
we define the norm
\[
\left\Vert f\right\Vert _{C_{b}^{k}}=\sum_{\operatorname{ord}D\le k}\left\Vert Df\right\Vert _{\infty}
\]
summing over all monomials in $X_{i}$, $1\le i\le5$ (which correspond
to left-invariant differential operators on $G$) of degree $\le k$.

We have the following formulation of Strömbergsson's
theorem:
\begin{thm}[Strömbergsson \cite{Strombergsson}]
\label{thm:StrombergssonThm} Let $\delta>0$ fixed, and let $\xi\in\mathbb{R}^{2}$
be a fixed Diophantine vector of type $\kappa\ge3/2$. Then for any $\nu\in C^{\infty}\left(\mathbb{R}\right)$ supported on $ \left[-L,L\right] $ $\left(L\ge1\right)$, $f\in C_{b}^{8}\left(\Gamma\backslash G\right)$
and $0<y<1,$ 
\begin{alignat}{1}
 &\label{eq:StrombergssonThm} \int_{\mathbb{R}}f\left(\Gamma\left(1,\xi\right)n_+(x)a\left(y\right)\right)\nu\left(x\right)\,\mbox{d}x\\ & =\int_{\Gamma\backslash G}f\,\mbox{d}\mu\int_{\mathbb{R}}\nu\left(x\right)\,\mbox{d}x
 +O\left(L\left\Vert \nu\right\Vert _{C^{2}}\left\Vert f\right\Vert _{C_{b}^{8}}y^{\min\left(1/2,2/\kappa\right)-\delta}\right).\nonumber 
\end{alignat}

\end{thm}

\subsection{Jacobi theta sums}\label{sec:Jacobi}

Recall the unique Iwasawa decomposition of a matrix $\begin{pmatrix}
a & b\\
c & d
\end{pmatrix}\in G'$:
\[
\begin{pmatrix}
a & b\\
c & d
\end{pmatrix}=n_{+}\left(x\right)a\left(\sqrt{y}\right)r\left(\phi\right)=\left(\tau,\phi\right)
\]
where $\tau=x+iy\in\mathfrak{H}$, $\phi\in[0,2\pi)$, and $\mathfrak{H}$
is the Poincaré half-plane model of the hyperbolic plane, which gives
an identification $G'=\mathfrak{H}\times[0,2\pi)$ with the left action
of $G'$ on $\mathfrak{H}\times[0,2\pi)$ given by 
\[
\begin{pmatrix}
a & b\\
c & d
\end{pmatrix}\left(\tau,\phi\right)=\left(\frac{a\tau+b}{c\tau+d},\phi+\arg\left(c\tau+d\right)\mbox{ mod}\,2\pi\right).
\]

For $ f\in \mathcal{S} (\mathbb{R}) $, let
\begin{alignat}{1}
\Theta_{f}\left(\tau,\phi;\xi\right) & =y^{1/4}\sum_{n\in\mathbb{Z}}f_{\phi}\left(\left(n-\xi_{2}\right)y^{1/2}\right)\label{eq:ThetaDef} e\left(\frac{1}{2}\left(n-\xi_{2}\right)^{2}x+n\xi_{1}\right) 
\end{alignat}
be the Jacobi Theta sum, where
\[
f_{\phi}\left(w\right)=\begin{cases}
f\left(w\right) & \phi=0\mod\,2\pi\\
f\left(-w\right) & \phi=\pi\mod\,2\pi\\
\left|\sin\phi\right|^{-1/2}\int_{\mathbb{R}}e\left[\frac{\frac{1}{2}\left(w^{2}+w'^{2}\right)\cos\phi-ww'}{\sin\phi}\right]f\left(w'\right)\,\mbox{d}w' & \phi\ne0\mod\,\pi.
\end{cases}
\]
It is well-known that the operators $U^{\phi}:f\mapsto f_{\phi}$
are unitary; note that $f_{\pi/2}=\hat{f}$. Moreover, $\Theta_{f}$ is a smooth function on $G$ (see \cite{MarklofCellarosi} for example). 

Let $\tilde{\Gamma}$ be the subgroup of $G$ defined by
\[
\tilde{\Gamma}=\left\{ \gamma = \left(\begin{pmatrix}
a & b\\
c & d
\end{pmatrix},\begin{pmatrix}
ab/2\\
cd/2
\end{pmatrix}+m\right):\begin{pmatrix}
a & b\\
c & d
\end{pmatrix}\in\Gamma',m\in\mathbb{Z}^{2}\right\} ,
\]
acting on $\mathfrak{H}\times[0,2\pi)\times \mathbb{R}^2$ by \[ \gamma (\tau , \phi ; \xi)= \left( \frac{a\tau +b}{c\tau + d}, \phi + \arg(c\tau + d) \mbox{ mod}\,2\pi ; \begin{pmatrix}
a & b\\
c & d 
\end{pmatrix} \xi + \begin{pmatrix}
ab/2\\
cd/2
\end{pmatrix}+ m\right). \]
It is well-known (see \cite[Proposition 4.9]{Marklof}) that $\left|\Theta_{f}\left(\tau,\phi;\xi\right)\right|^{2}$
is invariant under the left action of $\tilde{\Gamma}$. Moreover,
it is easy to see that
\begin{equation}
\int_{\mathbb{T}^{2}}\left|\Theta_{f}\left(\tau,\phi;\xi\right)\right|^{2}\,\mbox{d}\xi=\left\Vert f\right\Vert _{2}^{2}.\label{eq:ThetaL2Norm}
\end{equation}
We also have the following approximation for $y\ge1/2$ (see \cite[Proposition 4.10]{Marklof}):
For all $A > 1$
\begin{alignat}{1}
\left|\Theta_{f}\left(\tau,\phi;\xi\right)\right|^{2} & =y^{1/2}\sum_{m\in\mathbb{Z}}\left|f_{\phi}\left(y^{1/2}\left(m-\xi_{2}\right)\right)\right|^{2}+O\left(C_{A,f_\phi} y^{(-A+1)/2}\right)\label{eq:ThetaEstimate}\\
 & =y^{1/2}\left|f_{\phi}\left(y^{1/2}\left( m_0 -\xi_{2}\right) \right)\right|^{2}+O\left(C_{A,f_\phi} y^{(-A+1)/2}\right)\nonumber 
\end{alignat}
where $ m_0\in \mathbb{Z} $ so that $|m_0 - \xi_2| \le \frac{1}{2}$, $ C_{A,f_\phi} = \sup_{x\in \mathbb R} \left| \left (1+\left|x\right|\right)^{A} f_\phi(x) \right |^2$, and the error term is uniform in $x,\phi,\xi$.

\section{\label{sec:The-Smooth-Two-Point}The smooth pair correlation
function}

In order to prove Proposition \ref{prop:KeyProp}, we first notice
that $Q_{N}\left(\nu,h\right)$ can be interpreted as a sum of $\left|\Theta_{h}\right|^{2}$, namely

\[
Q_{N}\left(\nu,h\right)=\frac{1}{N}\sum_{k\ne0}\nu\left(\frac{k}{N}\right)\left|\Theta_{h}\left(\frac{k}{N}+\frac{i}{N^{2}},0,\begin{pmatrix}
0\\
\alpha
\end{pmatrix}\right)\right|^{2}.
\]

Next, we decompose $Q_{N}$ into a sum over the divisors $\sigma\mid N$, and show that the contribution of large $\sigma's$ is small. We first recall the following lemma from \cite{Marklof}:
\begin{lem}[{Marklof \cite[Lemma 6.6]{Marklof}}]\label{lem:MarklofLemma}
	Let $ \alpha  $ be Diophantine of type $ \kappa $, and $ f\in C(\mathbb{R}) $ be rapidly decreasing at $ \pm \infty $, $ f\ge 0 $. Then, for any fixed $ A>1 $ and $ 0 < \epsilon< 1/(\kappa -1 )  $ there exists $ k=k(A,\kappa,\epsilon ) $ such that
	\[ \sum_{d=1}^D \sum_{m\in \mathbb{Z}} f \left(T (d\alpha + m)\right) \ll \begin{cases} C_k T^{-A}
	& (D\le T^\epsilon) \\ C_2  & (T^\epsilon \le D \le T^{1/(\kappa-1)}) \\
	C_2 DT^{-1/(\kappa-1)} &  (D\ge T^{1/(\kappa -1 )})\end{cases}\]uniformly for all $ D,T>1 $, where $ C_n = \sup_{x\in \mathbb{R}}\left\{\left(1+\left|x\right|^{n}\right)f(x)\right\} $ (the implied constant is independent of $D,T,f$).
\end{lem}

We will use Lemma \ref{lem:MarklofLemma} to prove the following lemma:
\begin{lem}\label{lem:cuspEst}
Fix a Diophantine $\alpha$ of type $\kappa$,
and let $0<\delta<\frac{\kappa-1}{\kappa}$,  $ 0<\epsilon<\frac{\delta}{\kappa-1}$ be fixed. There exists $ n=n(\kappa, \delta,\epsilon)\in \mathbb N $  such that for any $\nu\in C^{\infty}\left(\mathbb{R}\right)$ supported on $ \left[-L,L\right] $ $\left(L\ge1\right)$,
$h\in C^{\infty}\left(\mathbb{R}\right)$ real valued and supported on $ \left[-1/2,3/2\right] $ and $ M \asymp N $,
\begin{alignat}{1}
Q_{N}\left(\nu,h\right) & =\frac{1}{N}\sum_{\begin{subarray}{c}
\sigma\mid N\\
\sigma\le M^{\delta}
\end{subarray}}\sum_{\left(k,N/\sigma\right)=1}\nu\left(\frac{k}{N/\sigma}\right)\left|\Theta_{h}\left(\frac{k}{N/\sigma}+\frac{i}{N^{2}},0,\begin{pmatrix}
0\\
\alpha
\end{pmatrix}\right)\right|^{2}\label{eq:TailLemmaEq}\\
 & +O\left(L \left\Vert h\right\Vert_{C^n}^2\left\Vert \nu\right\Vert _{\infty} N^{-\frac{\delta}{\kappa-1}+\epsilon}\right)\nonumber 
\end{alignat}
as $N\to\infty$.\end{lem}
\begin{proof}
Let $ \sigma \mid N $, and assume first that $N/\sigma$ is even. For $\left(k,N/\sigma\right)=1$, by the invariance of $\left|\Theta_{h}\right|^{2}$
under $\tilde{\Gamma}$ we have 
\begin{equation}
\left|\Theta_{h}\left(\frac{k}{N/\sigma}+\frac{i}{N^{2}},0,\begin{pmatrix}
0\\
\alpha
\end{pmatrix}\right)\right|^{2}=\left|\Theta_{h}\left(\frac{a}{N/\sigma}+\sigma^{2}i,\pi/2,\begin{pmatrix}
b\alpha\\
-k\alpha
\end{pmatrix}\right)\right|^{2}\label{eq:ThetaLift}
\end{equation}
where $\left(a,b\right)$ are (any) integers that solve the equation
$ak+bN/\sigma=-1$ such that $b$ is even. 
By (\ref{eq:ThetaLift}) and (\ref{eq:ThetaEstimate}), for $\left(k,N/\sigma\right)=1$
and for all $A > 1$ we have
\[
\left|\Theta_{h}\left(\frac{k}{N/\sigma}+\frac{i}{N^{2}},0,\begin{pmatrix}
0\\
\alpha
\end{pmatrix}\right)\right|^{2}=\sigma\sum_{m\in\mathbb{Z}}\left|\hat{h}\left(\sigma\left(m+k\alpha\right)\right)\right|^{2}+O\left(C_{A,\hat{h}}\sigma^{-A+1}\right),
\]
where $ C_{A,\hat{h}} = \sup_{x\in \mathbb R} \left| \left (1+\left|x\right|\right)^{A} \hat{h}(x) \right |^2$.

If $ N/\sigma $ is odd, then for $ (k,N/\sigma)=1 $, by the invariance of $ |\Theta_h|^2 $ under $ \tilde{\Gamma} $ we have\[ \left|\Theta_{h}\left(\frac{k}{N/\sigma}+\frac{i}{N^{2}},0,\begin{pmatrix}
0\\
\alpha
\end{pmatrix}\right)\right|^{2}=\left|\Theta_{h}\left(\frac{a}{N/\sigma}+\sigma^{2}i,\pi/2,\begin{pmatrix}
b\alpha\\
-k(\alpha+\frac{1}{2})
\end{pmatrix}\right)\right|^{2}\] where $\left(a,b\right)$ are (any) integers that solve the equation
$ak+bN/\sigma=-1$ such that $a$ is even, and hence, for all $A > 1$ we have
\begin{alignat*}{1}
\left|\Theta_{h}\left(\frac{k}{N/\sigma}+\frac{i}{N^{2}},0,\begin{pmatrix}
0\\
\alpha
\end{pmatrix}\right)\right|^{2}&=\sigma\sum_{m\in\mathbb{Z}}\left|\hat{h}\left(\sigma\left(m+k\left(\alpha+\frac{1}{2}\right)\right)\right)\right|^{2}\\
&+O\left(C_{A,\hat{h}}\sigma^{-A+1}\right).
\end{alignat*}

Thus, for all $ B > 1 $ there exists $ n=n(B,\delta)\in\mathbb{N} $ such that
\begin{alignat*}{1}
 & \frac{1}{N}\sum_{\begin{subarray}{c}
\sigma\mid N\\
\sigma>M^{\delta}
\end{subarray}}\sum_{\left(k,N/\sigma\right)=1}\nu\left(\frac{k}{N/\sigma}\right)\left|\Theta_{h}\left(\frac{k}{N/\sigma}+\frac{i}{N^{2}},0,\begin{pmatrix}
0\\
\alpha
\end{pmatrix}\right)\right|^{2}\\
 & \ll\left\Vert \nu\right\Vert _{\infty}\frac{1}{N}\sum_{\begin{subarray}{c}
\sigma\mid N\\
\sigma>M^{\delta}
\end{subarray}}\sigma\sum_{k\le LN/\sigma}\sum_{m\in\mathbb{Z}}\left|\hat{h}\left(\sigma\left(m+k\alpha_\sigma\right)\right)\right|^{2}+L\left\Vert \nu\right\Vert _{\infty}\left\Vert h\right\Vert^2 _{C_n}N^{-B},
\end{alignat*} where \[ \alpha_\sigma = \begin{cases} \alpha & N/\sigma \mbox{ is even} \\ \alpha+\frac{1}{2} & N/\sigma \mbox{ is odd}. \end{cases} \] To bound the main term, we divide the outer summation into three ranges: First, if $M^{\delta}<\sigma\le\left(LN\right)^{\frac{\kappa-1}{\kappa}}$,
then we have $LN/\sigma \ge \sigma^{\frac{1}{\kappa-1}}$, so by Lemma \ref{lem:MarklofLemma} we have
\[
\sum_{k\le LN/\sigma}\sum_{m\in\mathbb{Z}}\left|\hat{h}\left(\sigma\left(m+k\alpha_\sigma\right)\right)\right|^{2}\ll \left\Vert h\right\Vert_{C^1}^2 \frac{LN/\sigma}{\sigma^{\frac{1}{\kappa-1}}}
\]
and hence
\begin{alignat*}{1}
 & \frac{1}{N}\sum_{\begin{subarray}{c}
\sigma\mid N\\
M^{\delta}<\sigma\le\left(LN\right)^{\frac{\kappa-1}{\kappa}}
\end{subarray}}\sigma\sum_{k\le LN/\sigma}\sum_{m\in\mathbb{Z}}\left|\hat{h}\left(\sigma\left(m+k\alpha_\sigma\right)\right)\right|^{2}\\
 & \ll L \left\Vert h\right\Vert_{C^1}^2 \sum_{\begin{subarray}{c}
\sigma\mid N\\
M^{\delta}<\sigma\le\left(LN\right)^{\frac{\kappa-1}{\kappa}}
\end{subarray}}\sigma^{-\frac{1}{\kappa-1}}\ll L \left\Vert h\right\Vert_{C^1}^2 N^{\frac{-\delta}{\kappa-1}+\epsilon}.
\end{alignat*}

For $\left(LN\right)^{\frac{\kappa-1}{\kappa}} \le \sigma\le\left(LN\right)^{\frac{\kappa-1}{\kappa-\epsilon}}$
we have $\sigma^{\frac{1-\epsilon}{\kappa-1}}\le  LN/\sigma \le \sigma^{\frac{1}{\kappa-1}},$
so by Lemma \ref{lem:MarklofLemma},
\[
\sum_{k\le LN/\sigma}\sum_{m\in\mathbb{Z}}\left|\hat{h}\left(\sigma\left(m+k\alpha_\sigma\right)\right)\right|^{2}\ll \left\Vert h\right\Vert_{C^1}^2
\]
and then
\begin{alignat*}{1}
 & \frac{1}{N}\sum_{\begin{subarray}{c}
\sigma\mid N\\
\left(LN\right)^{\frac{\kappa-1}{\kappa}}\le \sigma\le\left(LN\right)^{\frac{\kappa-1}{\kappa-\epsilon}}
\end{subarray}}\sigma\sum_{k\le LN/\sigma}\sum_{m\in\mathbb{Z}}\left|\hat{h}\left(\sigma\left(m+k\alpha_\sigma\right)\right)\right|^{2}\\
 & \ll\frac{\left\Vert h\right\Vert_{C^1}^2}{N}\sum_{\begin{subarray}{c}
\sigma\mid N\\
\left(LN\right)^{\frac{\kappa-1}{\kappa}} \le \sigma\le\left(LN\right)^{\frac{\kappa-1}{\kappa-\epsilon}}
\end{subarray}}\sigma\ll  \frac{\left\Vert h\right\Vert_{C^1}^2 \left(LN\right)^{\frac{\kappa-1}{\kappa-\epsilon}+\epsilon}}{N}.
\end{alignat*}

For $\left(LN\right)^{\frac{\kappa-1}{\kappa-\epsilon}}\le\sigma,$ we have
$LN/\sigma \le \sigma^{\frac{1-\epsilon}{\kappa-1}}$, so by Lemma \ref{lem:MarklofLemma}
for every $A>1$ there exists $ n=n(A,\kappa,\epsilon)\in\mathbb{N} $ such that
\[
\sum_{k\le LN/\sigma}\sum_{m\in\mathbb{Z}}\left|\hat{h}\left(\sigma\left(m+k\alpha_\sigma\right)\right)\right|^{2}\ll\left\Vert h\right\Vert_{C^n}^2 \sigma^{-A},
\]
and hence for all $ B >1 $ there exists $ n=n(B,\kappa,\epsilon)\in\mathbb{N} $ such that
\[
\frac{1}{N}\sum_{\begin{subarray}{c}
\sigma\mid N\\
\sigma \ge \left(LN\right)^{\frac{\kappa-1}{\kappa-\epsilon}} 
\end{subarray}}\sigma\sum_{k\le LN/\sigma}\sum_{m\in\mathbb{Z}}\left|\hat{h}\left(\sigma\left(m+k\alpha_\sigma\right)\right)\right|^{2}\ll \left\Vert h\right\Vert_{C^n}^2 N^{-B}.
\]

The lemma now follows from the bounds in the different
ranges.

\end{proof}

We thus conclude that smooth averages of $Q_{N}$ can be approximated
by the following smooth sums of $\left|\Theta_{h}\right|^{2}$, which are closely related to the integral
on the l.h.s. of (\ref{eq:StrombergssonThm}), as we shall see in
\S\ref{sec:Equidistribution-of-Sums}:
\begin{cor}
\label{cor:MainSum}Fix a Diophantine $\alpha$ of type $\kappa$,
and let $\max\left(\frac{17}{18},1-\frac{2}{9\kappa}\right)<\eta \le 1$,  $0<\delta<\frac{\kappa-1}{\kappa}$, and  $ 0<\epsilon<\frac{\delta}{\kappa-1}$ be fixed. There exists $ n=n(\kappa,\delta,\epsilon)\in \mathbb N $ such that for any $\nu\in C^{\infty}\left(\mathbb{R}\right)$ supported on $ \left[-L,L\right] $ $\left(L\ge1\right)$ and
$\psi,h\in C^{\infty}\left(\mathbb{R}\right)$ real valued and supported on $ \left[-1/2,3/2\right] $,
\begin{alignat*}{1}
&\frac{1}{M^{\eta}}\sum_{N\in\mathbb{Z}}\psi\left(\frac{N-M}{M^\eta}\right)Q_{N}\left(\nu,h\right) =\frac{1}{M^{\eta}}\sum_{\sigma\le M^\delta}\frac{1}{\sigma}\sum_{N\in\mathbb{Z}}\psi\left(\frac{\sigma N-M}{M^\eta}\right)\frac{1}{N}\\
 & \times\sum_{\left(k,N\right)=1}\nu\left(\frac{k}{N}\right)\left|\Theta_{h}\left(\frac{k}{N}+\frac{i}{\sigma^2 N^{2}},0,\begin{pmatrix}
0\\
\alpha
\end{pmatrix}\right)\right|^{2}\\
 & +O\left(L \left\Vert h\right\Vert _{C^{n}}^{2} \left\Vert \nu\right\Vert _{\infty}\left\Vert \psi\right\Vert _{\infty}M^{-\frac{\delta}{\kappa-1}+\epsilon}\right)
\end{alignat*}
as $M\to\infty.$
\end{cor}

\section{\label{sec:Equidistribution-of-Sums}A geometric equidistribution result}

\subsection{Coordinates near the section}

Let $ 1/4\le v_0<v_1 $, and define the section
\[
H'_{v_0,v_1} =  \left\{ n_{+}\left(u\right)a\left(v\right): u\in\mathbb{R},\,v_0\le v \le v_1\right\}\subset G'.
\]

In order to prove Proposition \ref{prop:equi_prop}, given $ \epsilon>0 $ we define $H'_{\epsilon,v_{0},v_{1}}$ to be the following thickening of $H'_{v_0,v_1}$:
\[
H'_{\epsilon,v_{0},v_{1}}=\left\{ n_{+}\left(u\right)a\left(v\right)n_{-}\left(w\right):\,u\in\mathbb{R},\,v_{0}\le v\le v_{1},\,\left|w\right|\le\epsilon\right\},
\]
so $ (u,v,w) $ are local coordinates near the section
$H'_{v_0,v_1}$.
\begin{lem}
\label{Lem:FundamentalDomain}Fix $ 0<\epsilon<1/32 $, and let $ 1/4 \le v_0 < v_1 $. For any $\gamma\in\Gamma'$
such that $\gamma\notin\Gamma'_{\infty}$, we have
\[
\gamma H'_{\epsilon,v_{0},v_{1}}\cap H'_{\epsilon,v_{0},v_{1}}=\emptyset.
\]
\end{lem}
\begin{proof}
Let \[h=n_{+}\left(u\right)a\left(v\right)n_{-}\left(w\right),\,\tilde{h}=n_{+}\left(\tilde{u}\right)a\left(\tilde{v}\right)n_{-}\left(\tilde{w}\right)\in H'_{\epsilon,v_{0},v_{1}}.\] A short calculation yields that in the Iwasawa coordinates, $ h a(4\sqrt{2})$ has \[ y=\frac{32v^2}{1+(32w)^2} > 1, \]and likewise for $ \tilde{h}a(4\sqrt{2}) $ we have $ y>1 $. Therefore for any $\gamma\notin\Gamma'_{\infty}$,
$\gamma h a(4\sqrt{2})\ne \tilde{h} a(4\sqrt{2}),$ and hence $\gamma h \ne \tilde{h}.$
\end{proof}

Recall that we identify elements of $ G' $ with elements of $ G $ under the embedding $ g \mapsto (g,0) $. In particular $ n_+(x), n_-(x), a(y)$ and $r(\phi) $ are identified with $ (n_+(x),0), (n_-(x),0), (a(y),0) $ and $(r(\phi),0) $ respectively.

Let \[ H = \left\{ (1,\xi) n_+(u): (u,\xi)\in \mathbb{H}(\mathbb{R}) \right\} \simeq \mathbb{H}(\mathbb{R})\] and \[ \Gamma_H = \Gamma \cap H \simeq \mathbb{H}(\mathbb{Z}).\]
Note that $\Gamma_H \backslash H \simeq \Gamma \backslash \Gamma H$. Thus, for every $0<v_0< v_1$,
\begin{equation} (v_0,v_1)\times \mathbb{H}\left(\mathbb{Z}\right)\backslash \mathbb{H}\left(\mathbb{R}\right) \hookrightarrow \Gamma \backslash \Gamma H \{a(v):v_0< v<v_1\} \label{eq:embedding} \end{equation} is an embedding of $ (v_0,v_1) \times \mathbb{H}\left(\mathbb{Z}\right)\backslash \mathbb{H}\left(\mathbb{R}\right)$ as a submanifold of $ \Gamma \backslash G $. 

Let $H_{v_0,v_1}=H'_{v_0,v_1}\times\mathbb{R}^{2},$ and $H{}_{\epsilon,v_{0},v_{1}}=H'_{\epsilon,v_{0},v_{1}}\times\mathbb{R}^{2}$. The proof of the next lemma is similar to the proof of Lemma \ref{Lem:FundamentalDomain}:
\begin{lem}
\label{Lem:FundamentalDomain2}Fix $ 0<\epsilon<1/32 $, and let $ 1/4 \le v_0 < v_1 $. For any $\gamma\in\Gamma$ such
that $\gamma\notin \Gamma_H$, we have
\[
\gamma H_{\epsilon,v_{0},v_{1}}\cap H_{\epsilon,v_{0},v_{1}}=\emptyset.
\]

\end{lem}

The following key lemma shows that modulo $ \Gamma $, points on the lifted horocycle $ \left(1,\eta\right)a\left(\frac{1}{M}\right)r\left(-\frac{\pi}{2}\right)n_{-}\left(x\right) $ are in $ $$H_{\epsilon,v_0,v_1} $ exactly when they are close to the points in Proposition \ref{prop:equi_prop} (under the embedding (\ref{eq:embedding})) that we claim to be equidistributed:

\begin{lem}
	\label{Lem:KeyLemma}Fix $ \epsilon>0 $, and let $ 1/4 \le v_0 < v_1 $. For any $ M>0, \eta \in \mathbb{R}^2, x\in\mathbb{R} $, we have
	\[ \Gamma\left(1,\eta\right)a\left(\frac{1}{M}\right)r\left(-\frac{\pi}{2}\right)n_{-}\left(x\right)\cap H_{\epsilon,v_{0},v_{1}}\neq\emptyset \] if and only if $x=M^{2}d/c+w,$ where $\left(c,d\right)=1$, $v_{0}\le M/c\le v_{1}$
	and $\left|w\right|\le\epsilon$, and then
	
	\begin{alignat*}{1} & \Gamma\left(1,\eta\right)a\left(\frac{1}{M}\right)r\left(-\frac{\pi}{2}\right)n_{-}\left(x\right) \\
		& = \Gamma \left(1,\begin{pmatrix}
			a & b\\
			c & d
		\end{pmatrix} \eta \right) n_+\left(\frac{a}{c}\right)a\left(\frac{M}{c}\right) n_-\left(w\right) 
	\end{alignat*}
	
	where  $\left(a,b\right)$ are any integers that
	solve the equation $ad-bc=1$.
\end{lem}

\begin{proof}
	For any $g=\begin{pmatrix}
		\alpha & \beta \\
		\gamma & \delta
	\end{pmatrix}\in G'$, we have
	\[ g= n_{+}\left(u\right)a\left(v\right)n_{-}\left(w\right)= \begin{pmatrix}
	v+uw/v & u/v\\
	w/v & 1/v
	\end{pmatrix} \] if and only if $ \beta/\delta = u, 1/\delta = v, \gamma/\delta = w $.
	Hence $ g\in H'_{\epsilon,v_0,v_1} $ if and only if $ v_0\ \le 1/\delta \le v_1 $ and $ |\gamma/\delta|\le \epsilon $.
	
	Let $\begin{pmatrix}
	a & b \\
	c & d
	\end{pmatrix}\in \Gamma'$. Then \[ \begin{pmatrix}
	a & b \\
	c & d
	\end{pmatrix} a\left(\frac{1}{M}\right)r\left(-\frac{\pi}{2}\right)n_{-}\left(x\right) =
	 \begin{pmatrix}
	 ax/M-bM & a/M \\
	 cx/M-dM & c/M
	 \end{pmatrix} \]is in $ H'_{\epsilon,v_0,v_1} $ if and only if $ v_0 \le M/c \le v_1 $ and $ \left|x- M^2d/c\right|\le \epsilon $. Moreover, $ u=a/c, v=M/c $ and $ w=x-M^2d/c $. Since the same calculation extends to $ G $, the statement of the lemma follows.
\end{proof}

\subsection{Proof of Proposition \ref{prop:equi_prop}}

We will now use Theorem \ref{thm:StrombergssonThm} to prove Proposition \ref{prop:equi_prop}.

\begin{proof}
Fix a Diophantine vector $ \eta $ of type $\kappa$, $\delta>0$. In addition, fix  $0<\epsilon < 1/32$, and fix $\omega\in C^{\infty}\left(\mathbb{R}\right)$ supported on $ \left[-\epsilon,\epsilon\right] $
such that $\int\omega\left(w\right)\,\mbox{d}w=1$. Let $ 1/4 \le v_0 < v_1 $, 
$f\in C^{8}\left((0,\infty) \times \mathbb{H}\left(\mathbb{Z}\right) \backslash \mathbb{H}\left(\mathbb{R}\right)\right)
$ supported on $ [v_0,v_1] \times \mathbb{H}\left(\mathbb{Z}\right) \backslash \mathbb{H}\left(\mathbb{R}\right) $ and $\nu\in C^{\infty}\left(\mathbb{R}\right)$ supported on $ \left[-L,L\right] $ $\left(L\ge1\right)$.

Recalling the embedding (\ref{eq:embedding}), we define
\begin{gather*}
\tilde{F}\left(g\right)=\begin{cases}
f\left(v,u,\xi\right)\omega\left(w\right) & g=\left(1,\xi\right)n_{+}\left(u\right)a\left(v\right)n_{-}\left(w\right) \in H_{\epsilon,v_{0},v_{1}}\\
0 & \mbox{otherwise}
\end{cases}
\end{gather*}
and let $F\in C_b^{8}\left(\Gamma\backslash G\right)$ be defined
by
\begin{equation}
F\left(\Gamma g\right)=\sum_{\gamma\in \Gamma_H \backslash\Gamma}\tilde{F}\left(\gamma g\right).\label{eq:periodize_F}
\end{equation}
Note that by Lemma \ref{Lem:FundamentalDomain2}, all but one of the
terms in (\ref{eq:periodize_F}) vanish, so $F$ is well-defined. 

We have
\begin{alignat*}{1}
 & I:=\int_\mathbb{R} F\left(\Gamma\left(1,\eta\right)n_+(x)a\left(\frac{1}{M}\right)r\left(-\frac{\pi}{2}\right)\right)\nu\left(x\right)\,\mbox{d}x\\
 & =\frac{1}{M^{2}}\int_\mathbb{R} F\left(\Gamma\left(1,\eta\right)a\left(\frac{1}{M}\right)r\left(-\frac{\pi}{2}\right)n_{-}\left(x\right)\right)\nu\left(-\frac{x}{M^{2}}\right)\,\mbox{d}x.
\end{alignat*}
Note that since $ v_0\ge1/4 $ and $ 0 < \epsilon <1/32 $, intervals of the form $ d/c+w/M^2 $ are disjoint for $ v_0 \le M/c \le v_1, |w|\le \epsilon$. Thus, by Lemma \ref{Lem:KeyLemma}, 
\begin{alignat*}{1}
I & =\frac{1}{M^{2}}\sum_{\begin{subarray}{c}
c,d\\
\left(c,d\right)=1
\end{subarray}}f\left(\frac{M}{c}, \frac{a}{c},\begin{pmatrix}
a & b\\
c & d
\end{pmatrix} \eta \right)\int_\mathbb{R} \nu\left(-d/c-w/M^{2}\right)\omega\left(w\right)\,\mbox{d}w\\
 & =\frac{1}{M^{2}}\sum_{\begin{subarray}{c}
c,d\\
\left(c,d\right)=1
\end{subarray}}\nu\left(\frac{-d}{c}\right)f\left(\frac{M}{c}, \frac{a}{c},\begin{pmatrix}
a & b\\
c & d
\end{pmatrix} \eta \right)+O\left(L\left\Vert \nu\right\Vert _{C^{1}} \left\Vert f\right\Vert_\infty M^{-2}\right).
\end{alignat*}

In terms of the Iwasawa coordinates (Section \ref{sec:Jacobi}), the measure $ \mu $ is given by \[ \mbox{d} \mu (\tau , \phi; \xi) = \frac{3}{\pi^{2}}y^{-2} \, \mbox{d}x \, \mbox{d}y \, \mbox{d}\phi \, \mbox{d}\xi.\]Therefore,
\[ \int_{\Gamma\backslash G}F\,\mbox{d}\mu= \frac{3}{\pi^{2}} \int\limits_{[0,1]^2}   \int_{-\frac{\pi}{2}}^{\frac{\pi}{2}} \int_0^\infty   \int_{0}^1 \tilde{F}\left((1,\xi)n_+(x) a(\sqrt{y}) r(\phi)\right) y^{-2} \,  \mbox{d}x \, \mbox{d}y \, \mbox{d}\phi \, \mbox{d}\xi.\]

Making the change of variables $ x=u+\frac{wv^2}{1+w^2} $, $ y=\frac{v^2}{1+w^2} $, $ \phi = \arctan w $ for which the Jacobian is equal to \[ \frac{2v}{(1+w^2)^2}=2y^2v^{-3}\] so that \[ (1,\xi)n_+(x) a(\sqrt{y}) r(\phi) = \left(1,\xi\right)n_{+}\left(u\right)a\left(v\right)n_{-}\left(w\right), \] we get (using the fact that $ \omega $ is of unit mass) that
\[
\int_{\Gamma\backslash G}F\,\mbox{d}\mu=\frac{6}{\pi^{2}}\int_0^\infty \int_{\mathbb{H}(\mathbb{Z}) \backslash \mathbb{H}(\mathbb{R})} f\left(v,u,\xi\right) v^{-3}\,\mbox{d}\mu_{H}\mbox{d}v.\] 
Thus, by Theorem \ref{thm:StrombergssonThm},
\begin{alignat}{1}\label{eq:EquiFarey} 
 & \frac{1}{M^{2}}\sum_{\begin{subarray}{c}
 		c,d\\
 		\left(c,d\right)=1
 	\end{subarray}}\nu\left(-\frac{d}{c}\right)f\left(\frac{M}{c}, \frac{a}{c},\begin{pmatrix}
 	a & b\\
 	c & d
 \end{pmatrix} \eta \right)  \\
 & =\frac{6}{\pi^{2}} \hat{\nu}(0) \int_0^\infty \int_{\mathbb{H}(\mathbb{Z}) \backslash \mathbb{H}(\mathbb{R})} f\left(v,u,\xi\right) v^{-3}\,\mbox{d}\mu_{H}\mbox{d}v + \nonumber \\
 & O\left(L\left\Vert \nu\right\Vert _{C^{2}}\left\Vert F\right\Vert _{C_{b}^{8}}M^{-\min\left(1/2,2/\kappa\right)+\delta}\right).\nonumber 
\end{alignat}
Finally, we note that
in the coordinates of $H_{\epsilon,v_{0},v_{1}},$
\begin{eqnarray*}
	X_{1} & = & v^{2}\frac{\partial}{\partial u}-wv\frac{\partial}{\partial v}-w^{2}\frac{\partial}{\partial w},\hspace{1em}X_{2}=\frac{\partial}{\partial w}\\
	X_{3} & = & v\frac{\partial}{\partial v}+2w\frac{\partial}{\partial w},\hspace{1em}X_{4}=\frac{v^2+uw}{v}\frac{\partial}{\partial\xi_{1}}+\frac{w}{v}\frac{\partial}{\partial\xi_{2}}\\
	X_{5} & = & \frac{u}{v}\frac{\partial}{\partial\xi_{1}}+\frac{1}{v}\frac{\partial}{\partial\xi_{2}},
\end{eqnarray*}
and therefore (Since $ \omega $ and $ \epsilon $ are fixed) $\left\Vert F\right\Vert _{C_{b}^{8}} \ll \left\Vert f\right\Vert _{C^{8}}$. Note that equation (\ref{eq:EquiFarey}) also holds with $ \nu\left( \frac{d}{c}\right) $ on the l.h.s. instead of $ \nu\left(- \frac{d}{c}\right) $, since we can replace the function $ \nu(x) $ with $ \nu(-x) $, leaving the r.h.s. of  (\ref{eq:EquiFarey}) unchanged. Hence, the statement of the proposition follows.
\end{proof}

\subsection{Proof of Proposition \ref{prop:KeyProp}}
\label{sec:Prop2.2Proof}
\begin{proof}
By the invariance of $\left|\Theta_{h}\right|^2$
under $\tilde{\Gamma}$ and since for any $ \begin{pmatrix}
a & b\\
c & d
\end{pmatrix} \in \Gamma' $ we have $ ab/2 \equiv (a+b-1)/2 \mod 1$, we see that the function
\[ F(x+iy,\phi,\xi) = \left|\Theta_{h}\left(x+iy,\phi,\begin{pmatrix}
\xi_{1}-\frac{1}{2}\\
\xi_{2}-\frac{1}{2}
\end{pmatrix}\right)\right|^{2} \] belongs to $ C^\infty(\Gamma \backslash G) $, i.e., for every $ \left(\begin{pmatrix}
a & b\\
c & d
\end{pmatrix},m\right)\in \Gamma$, we have
\begin{equation}
F(\tau, \phi; \xi) = F\left( \frac{a\tau +b}{c\tau + d}, \phi + \arg(c\tau + d) \mbox{ mod} \, 2\pi ; \begin{pmatrix}
a & b\\
c & d 
\end{pmatrix} \xi + m\right)\label{eq:InvarianceEq}.\end{equation}
In particular, $ F $ is invariant under $ \Gamma_H $, and therefore for fixed $y,\phi$, the function
\[
f_{y,\phi}\left(x,\xi\right)= F(x+iy,\phi,\xi)
\]
belongs to $C^{\infty}\left(\mathbb{H}\left(\mathbb{Z}\right)\backslash\mathbb{H}\left(\mathbb{R}\right)\right)$,
and by (\ref{eq:ThetaL2Norm})
\[
\int_{\mathbb{H}\left(\mathbb{Z}\right)\backslash\mathbb{H}\left(\mathbb{R}\right)} f_{y,\phi}\,d\mu_{H}=\left\Vert h\right\Vert _{2}^{2}.
\]
Moreover, by the same reasoning of (\ref{eq:ThetaEstimate}), for $y\ge 1/2$
\begin{equation}
f_{y,\frac{\pi}{2}}\left(x,\xi\right)=y^{1/2}\left|\hat{h}\left(y^{1/2}\left(m_0- \xi_{2}+\frac{1}{2}\right) \right)\right|^{2}+R\left(x,y,\xi_1, \xi_2 \right)\label{eq:Theta_est}
\end{equation}
where $ m_0\in \mathbb{Z} $ so that $ \left| m_0-\xi_2+\frac{1}{2} \right| \le \frac{1}{2}$; for all $A>1,k\in\mathbb{N}$, there exists $ n=n(A,k) $ such that $R$ and all its partial derivatives of order $ k $ are uniformly
bounded by $O\left(\left\Vert h\right\Vert _{C^{n}}^2 y^{-A}\right)$.
Fix $ \sigma \ge 1$ and let
\[ f(v,u,\xi) = f_{\sigma^2,\frac{\pi}{2}}(u,\xi) \sigma ^{-1}v \psi\left(M^{1-\eta}\left(\sigma v^{-1}-1 \right) \right)
. \] Since $ \psi $ is supported on $[-1/2,3/2]$ and $ \eta \le 1 $, $ f $  vanishes unless  \[ -1/2 \le \sigma v^{-1} -1 \le 3/2, \] and in particular it  vanishes unless $ 2/5 \le v \le 2\sigma $, so $ f $ is compactly supported on $[1/4, \infty) \times \mathbb{H}\left(\mathbb{Z}\right) \backslash\mathbb{H}\left(\mathbb{R}\right)$.
If $ (k,N)=1 $ and $ a,b $ are integers that solve the equation $ ak+bN = -1 $, then by (\ref{eq:InvarianceEq}),
\begin{alignat*}{1}
&\left|\Theta_{h}\left(\frac{k}{N}+\frac{i}{\sigma^2 N^{2}},0,\begin{pmatrix}
0\\
\alpha
\end{pmatrix}\right)\right|^{2} = F\left(\frac{k}{N}+\frac{i}{\sigma^2 N^2}, 0, \begin{pmatrix}
\frac 1 2\\
\alpha+ \frac 1 2\end{pmatrix} \right) \\
& = F\left(\frac{a}{N}+i\sigma^2, \frac{\pi}{2}, \begin{pmatrix}
	a & b\\
	N & -k
\end{pmatrix} \begin{pmatrix}
	\frac 1 2\\
	\alpha+ \frac 1 2\end{pmatrix} \right) \\
& = f_{\sigma^2,\frac \pi 2}\left( \frac a N, \begin{pmatrix}
	a & b\\
	N & -k
\end{pmatrix} \begin{pmatrix}
	\frac 1 2\\
	\alpha+ \frac 1 2\end{pmatrix}\right). \end{alignat*}
Thus, 
\begin{alignat}{1}\label{eq:equi_eq}
 & \frac{1}{M}\sum_{\sigma\le M^\delta}\frac{1}{\sigma}\sum_{N\in\mathbb{Z}}\psi\left(\frac{\sigma N-M}{M^\eta}\right) \frac{1}{N} \\
 & \times  \sum_{\left(k,N\right)=1}\nu\left(\frac{k}{N}\right)\left|\Theta_{h}\left(\frac{k}{N}+\frac{i}{\sigma^2 N^{2}},0,\begin{pmatrix}
0\\
\alpha
\end{pmatrix}\right)\right|^{2} \nonumber \\
 &= \sum_{\sigma\le M^\delta}\frac{1}{M^{2}}\sum_{\begin{subarray}{c}
c,d\\ 
\left(c,d\right)=1
\end{subarray}}\nu\left(\frac{-d}{c}\right) f\left(\frac{M}{c}, \frac{a}{c},\begin{pmatrix}
 	a & b\\
 	c & d
 \end{pmatrix} \begin{pmatrix}
 \frac 1 2\\
 \alpha+ \frac 1 2 \end{pmatrix} \right),\nonumber
\end{alignat}
where $ (a,b) $ are any integers that solve the equation $ ad-bc=1$.

By Proposition \ref{prop:equi_prop} (applied with $ \nu (-x) $; recall the remark in the end of the proof of Proposition \ref{prop:equi_prop}), the right hand side of (\ref{eq:equi_eq}) is equal to
\begin{alignat*}{1}
&\frac{6}{\pi^{2}}\hat{\nu}\left(0\right)\left\Vert h\right\Vert _{2}^{2}\hat{\psi}(0)M^{-1+\eta}\sum_{\sigma\le M^\delta} \frac{1}{\sigma^2} \\
 & +O\left(L\left\Vert \nu\right\Vert _{C^{2}}M^{-\min\left(1/2,2/\kappa\right)+\delta}\sum_{\sigma\le M^\delta} \left\Vert f\right\Vert _{C^{8}}\right).
\end{alignat*}
The main term is equal to
\begin{alignat*}{1}
\hat{\nu}\left(0\right)\left\Vert h\right\Vert _{2}^{2}\hat{\psi}(0)M^{-1+\eta} +O\left(\hat{\nu}\left(0\right)\left\Vert h\right\Vert _{2}^{2}\hat{\psi}(0)M^{-1+\eta-\delta}\right).
\end{alignat*}

On the other hand, $ \left\Vert 
f\right\Vert _{C^{8}}=O\left(M^{8(1-\eta)}\left\Vert h\right\Vert _{C^{n}}^2 \left\Vert \psi \right\Vert _{C^{8}} \sigma\right).$
Thus, 
\begin{alignat*}{1}
 & \frac{1}{M^{\eta}}\sum_{\sigma\le M^\delta}\frac{1}{\sigma}\sum_{N\in\mathbb{Z}}\psi\left(\frac{\sigma N-M}{M^\eta}\right)\frac{1}{N}\\
 & \times \sum_{\left(k,N\right)=1}\nu\left(\frac{k}{N}\right)\left|\Theta_{h}\left(\frac{k}{N}+\frac{i}{\sigma^2 N^{2}},0,\begin{pmatrix}
0\\
\alpha
\end{pmatrix}\right)\right|^{2}=\hat{\psi}(0)\hat{\nu}\left(0\right)\left\Vert h\right\Vert _{2}^{2} \\
 & +O\biggl(\hat{\nu}\left(0\right)\left\Vert h\right\Vert _{2}^{2}\hat{\psi}(0)M^{-\delta} + L\left\Vert h\right\Vert _{C^{n}}^2 \left\Vert \nu\right\Vert _{C^{2}} \left\Vert \psi \right\Vert _{C^{8}} M^{9(1-\eta)-\min\left(1/2,2/\kappa\right)+3\delta}\biggr).
\end{alignat*}

Recalling Corollary \ref{cor:MainSum}, Proposition \ref{prop:KeyProp}
now follows from the condition  $\max\left(\frac{17}{18},1-\frac{2}{9\kappa}\right)<\eta \le 1$, assuming $ \delta $ is chosen sufficiently small.
\end{proof}

\section{\label{sec:Proof-of-the}Proof of the main theorems}

We are now able to prove the main theorems using Proposition \ref{prop:KeyProp}.

\subsection{Proof of Theorem \ref{thm:MainThmSmooth}}
\begin{proof}
Fix $\max\left(17/18,1-2/9\kappa\right)<\eta \le 1$.  By Proposition \ref{prop:KeyProp}, there exist $ s=s(\kappa,\eta)>0, n=n(\kappa,\eta)\in \mathbb{N} $ such that for for any $\nu\in C^{\infty}\left(\mathbb{R}\right)$ supported on $ \left[-L,L\right] $ $\left(L\ge1\right)$ and
$\psi,h\in C^{\infty}\left(\mathbb{R}\right)$ real valued and supported on $ \left[-1/2,3/2\right] $,
\begin{alignat}{1}\label{eq:proof2.2a}
\frac{1}{M^{\eta}}\sum_{N\in\mathbb{Z}}\psi\left(\frac{N-M}{M^\eta}\right)Q_{N}\left(\nu,h\right)& =\hat{\psi}(0)\hat{\nu}\left(0\right)\left\Vert h\right\Vert _{2}^{2}\\
& +O\left(L\left\Vert h\right\Vert _{C^{n}}^2\left\Vert \nu\right\Vert _{C^{2}} \left\Vert \psi \right\Vert _{C^{8}} M^{-s}\right). \nonumber
\end{alignat}
as $M\to\infty.$
By Poisson summation formula,
\begin{alignat}{1}\label{eq:proof2.2b}
&\frac{1}{M^{\eta}}\sum_{N\in\mathbb{Z}}\psi\left(\frac{N-M}{M^\eta}\right)R_{2,N}\left(\hat{\nu},h;\varphi_N\right) \\ &=\frac{1}{M^{\eta}}\sum_{N\in\mathbb{Z}}\psi\left(\frac{N-M}{M^\eta}\right)
  \left(\nu\left(0\right)\frac{1}{N^{2}}\left|\sum_{n\in\mathbb{Z}}h\left(\frac{n-\alpha}{N}\right)\right|^{2}+Q_{N}\left(\nu,h\right)\right)\nonumber 
\end{alignat}
Note that
\begin{equation}
\label{eq:proof2.2c}
\frac{1}{N^{2}}\left|\sum_{n\in\mathbb{Z}}h\left(\frac{n-\alpha}{N}\right)\right|^{2}=\left|\hat{h}(0)\right|^{2}+O\left(\left\Vert h\right\Vert _{C^1}^2N^{-1}\right)
\end{equation}
and
\begin{alignat}{1}\label{eq:proof2.2d}
\frac{1}{M^{\eta}}\sum_{N\in\mathbb{Z}}\psi\left(\frac{N-M}{M^\eta}\right) & =\hat{\psi}(0)+O\left(\left\Vert \psi'\right\Vert _{\infty}M^{-\eta}\right). 
\end{alignat}

Theorem \ref{thm:MainThmSmooth} now follows by substituting (\ref{eq:proof2.2a}), (\ref{eq:proof2.2c}) and (\ref{eq:proof2.2d}), into (\ref{eq:proof2.2b}).
\end{proof}

\subsection{Proof of Theorem \ref{thm:MainThm}}
\begin{proof}
Fix $\max\left(17/18,1-2/9\kappa\right)<\eta \le 1$. By Theorem \ref{thm:MainThmSmooth},
there exist $ t=t(\kappa,\eta)>0, n=n(\kappa,\eta)\in\mathbb{N} $ such that for any $\nu\in C^{\infty}\left(\mathbb{R}\right)$ supported on $ \left[-L,L\right] $ $\left(L\ge1\right)$ and
$\psi,h\in C^{\infty}\left(\mathbb{R}\right)$ real valued and supported on $ \left[-1/2,3/2\right] $,
\begin{alignat}{1}\label{eq:Unsmoothing}
	\frac{1}{M^{\eta}}\sum_{N\in\mathbb{Z}}\psi\left(\frac{N-M}{M^\eta}\right)R_{2,N}\left(\hat{\nu},h;\varphi_N\right) & =\hat{\psi}(0) \left(\nu\left(0\right)\left|\hat{h}(0)\right|^{2}+\hat{\nu}\left(0\right)\left\Vert h\right\Vert _{2}^{2}\right)\\
	& +O\left(L\left\Vert h\right\Vert _{C^{n}}^2\left\Vert \nu\right\Vert \nonumber _{C^{2}}  \left\Vert \psi \right\Vert _{C^{8}} M^{-t}\right)
\end{alignat}
as $M\to\infty$.

Fix a bounded interval $A$ and denote by $\chi_A$ its characteristic function. Given $\epsilon>0$, we can find smooth functions $\nu_{\epsilon,\pm}$
such that:
\begin{enumerate}[(i)]
\item $\nu_{\epsilon,\pm}$ are supported in $\left[-\epsilon^{-1},\epsilon^{-1}\right].$

\item $\hat{\nu}_{\epsilon,-}\le\chi_A\le\hat{\nu}_{\epsilon,+}$.

\item $\nu_{\epsilon,\pm}\left(0\right)=|A|+O\left(\epsilon\right).$

\item $\hat{\nu}_{\epsilon,\pm}\left(0\right)=\chi_A\left(0\right)+O\left(\epsilon\right).$

\item The derivative of any order $ (k\ge0) $ of $\nu_{\epsilon,\pm}$ is uniformly
bounded.

\end{enumerate}
See \cite[\S8]{Marklof} for a detailed construction.

Let $ \delta >0$, and choose $\psi_{\pm},h_{\pm}$ to be a smooth approximation to the characteristic function on the interval $\left[0,1\right]$
such that:
\begin{enumerate}[(i)] \item $0\le \psi_{\pm},h_{\pm}\le1.$ 
	
	\item $\psi_{\pm},h_{\pm}=1$ in $\left[M^{-\delta},1-M^{-\delta}\right]$.
	
	\item $\psi_{\pm},h_{\pm}=0$ in the complement of $\left[-M^{-\delta},1+M^{-\delta}\right]$. 
	
	\item  $ \psi_{\pm}^{(k)},h_{\pm}^{(k)}=O(M^{\delta k}) $ for the $ k $-th derivative $ (k\ge0) $.
\item  the following inequalities hold:

\begin{alignat*}{1}
&\frac{1}{M^{\eta}}\sum_{N\in\mathbb{Z}}\psi_{-}\left(\frac{N-M}{M^\eta}\right)R_{2,N}\left(\hat{\nu}_{\epsilon,-},h_{-};\varphi_N\right) \\
& \le \frac{1}{M^{\eta}}\sum_{M\le N\le M+M^{\eta}}R_{2,N}\left(A;\varphi_N\right)  + \chi_A(0) \\
& 
  \le\frac{1}{M^{\eta}}\sum_{N\in\mathbb{Z}}\psi_{+}\left(\frac{N-M}{M^\eta}\right)R_{2,N}\left(\hat{\nu}_{\epsilon,+},h_{+};\varphi_N\right).
\end{alignat*}

\end{enumerate}

Note that $\hat{\psi}_{\pm}(0)$, $\left|\hat{h}_{\pm}(0)\right|^{2}$, and $\left\Vert h_{\pm}\right\Vert _{2}^{2}$ are all of the form $1+O\left(M^{-\delta}\right).$
 
By (\ref{eq:Unsmoothing}),
\begin{alignat*}{1}
&\frac{1}{M^{\eta}}\sum_{N\in\mathbb{Z}}\psi_{\pm}\left(\frac{N-M}{M^\eta}\right)R_{2,N}\left(\hat{\nu}_{\epsilon,\pm},h_{\pm};\varphi_N\right)  =\nu_{\epsilon,\pm}\left(0\right)+\hat{\nu}_{\epsilon,\pm}\left(0\right)\\
& +O\left(M^{-\delta}+\epsilon^{-1} M^{-t+(2n+8)\delta}\right) =  |A|+ \chi_{A}(0)\\
& + O\left(M^{-\delta} + \epsilon+\epsilon^{-1}M^{-t+(2n+8)\delta}\right),
\end{alignat*}
and Theorem \ref{thm:MainThm} follows with $ s = t/(2n+10) $ by choosing $ \delta = t/(2n+10) $, and $ \epsilon= M^{-t/(2n+10)} $.
\end{proof}

\appendix

\section{\label{sec:Uniform Distribution of}Uniform distribution mod 1 of \texorpdfstring{$\varphi_N(n)$}{\texttheta n}}

We show that $ \varphi_N(n) $
is uniformly distributed mod $1$ (for any $\alpha$), and give an
explicit rate of decay (uniform in $ \alpha $) for its discrepancy
\[
D_{N}=\sup_{0\le a<b\le1}\left|\frac{1}{N}\sum_{n=1}^{N}\chi_{[a,b]}\left(\left\{\frac{\left(n-\alpha\right)^{2}}{2N}\right\}\right)-(b-a)\right|.
\]

\begin{prop}
	We have 
	\[
	D_{N}=O\left(\frac{\log N \left(\sqrt{\log N}+\tau(N)\right)}{\sqrt{N}}\right).
	\]
	where $ \tau(n) = \sum_{d\mid n} 1 $ is the divisor function, and the implied constant is independent of $ \alpha $.
\end{prop}
\begin{proof}
	By Erd\H{o}s\textendash Turán inequality, there exists a constant
	$C>0$ such that for every integer $m$
	\[
	D_{N}\le C\left(\frac{1}{m}+\frac{1}{N}\sum_{k=1}^{m}\frac{1}{k}\left|\sum_{n=1}^{N}e\left(\frac{k\left(n-\alpha\right)^{2}}{2N}\right)\right|\right).
	\]
	By Weyl differencing,
	\begin{alignat}{1}
		\left|\sum_{n=1}^{N}e\left(\frac{k\left(n-\alpha\right)^{2}}{2N}\right)\right|^{2} & =N+2\mbox{Re}\left[\sum_{h=1}^{N-1}e\left(\frac{k}{2N}\left(h^{2}-2h\alpha\right)\right)\sum_{n=1}^{N-h}e\left(\frac{kh}{N}n\right)\right].\label{eq:WeylDiff}
	\end{alignat}
	The inner sum of (\ref{eq:WeylDiff}) is equal to $N-h$ whenever
	$\frac{N}{\left(k,N\right)}\mid h$; otherwise it is equal to $\frac{1-e\left(-\frac{kh^{2}}{N}\right)}{1-e\left(\frac{kh}{N}\right)}$.
	It follows that
	\begin{alignat*}{1}
		\left|\sum_{n=1}^{N}e\left(\frac{k\left(n-\alpha\right)^{2}}{2N}\right)\right|^{2} & \le N\left(k,N\right)+4\sum_{\begin{subarray}{c}
				1\le h<N\\
				\frac{N}{\left(k,N\right)}\nmid h
			\end{subarray}}\frac{1}{\left|1-e\left(\frac{kh}{N}\right)\right|}.
		\end{alignat*}
		We have
		\begin{alignat*}{1}
			\sum_{\begin{subarray}{c}
					1\le h<N\\
					\frac{N}{\left(k,N\right)}\nmid h
				\end{subarray}}\frac{1}{\left|1-e\left(\frac{kh}{N}\right)\right|} & =\sum_{\begin{subarray}{c}
				-\frac{N}{2\left(k,N\right)}<a\le\frac{N}{2\left(k,N\right)}\\
				a\ne0
			\end{subarray}}\frac{1}{\left|1-e\left(\frac{a}{N/\left(k,N\right)}\right)\right|}\sum_{\begin{subarray}{c}
			1\le h<N\\
			\frac{k}{\left(k,N\right)}h\equiv a\,\left(\frac{N}{\left(k,N\right)}\right)
		\end{subarray}}1\\
		& \le\frac{1}{4}N\sum_{\begin{subarray}{c}
				-\frac{N}{2\left(k,N\right)}<a\le\frac{N}{2\left(k,N\right)}\\
				a\ne0
			\end{subarray}}\frac{1}{\left|a\right|}\ll N\log N
		\end{alignat*}
		and therefore 
		\begin{alignat*}{1}
			D_{N} & \ll\frac{1}{m}+\frac{1}{\sqrt{N}}\sum_{k=1}^{m}\frac{\sqrt{\log N}+\sqrt{\left(k,N\right)}}{k} \\ 
			& \ll \frac{1}{m}+\frac{\sqrt{\log N} \log m}{\sqrt{N}}+\frac{1}{\sqrt{N}}\sum_{\sigma\mid N}\frac{1}{\sqrt{\sigma}}\sum_{\begin{subarray}{c}
					1\le k\le m/\sigma\\
					\left(k,N/\sigma\right)=1
				\end{subarray}}\frac{1}{k} \\
				& \ll\frac{1}{m}+\frac{\sqrt{\log N}+ \tau(N)}{\sqrt{N}}\log m\ll\frac{\log N \left(\sqrt{\log N}+\tau(N)\right)}{\sqrt{N}}
			\end{alignat*}
			choosing $m=N$.\end{proof}
		
\section{\label{sec:MoreGeneral}More general pair correlation functions}

We consider slightly more general pair correlation measures, which also consider a general truncation of the indices $i,j$: For bounded intervals $A,B_1,B_2\subset\mathbb{R}$ set
\begin{multline}\label{generalpair}
	R_{2,N}(A,B_1,B_2;\varphi) \\
	= \frac{\#\left\{ (i,j)\in {\mathbb Z}^2\mid i\neq j , \; \varphi(i)-\varphi(j) \in N^{-1} A + \mathbb{Z},\  (\frac{i}{N}, \frac{j}{N})\in B_1\times B_2 \right\}}{N} .
\end{multline}

Our original pair correlation measure \eqref{eq:R2Def} is recovered by setting $B_1=B_2=(0,1]$. We have the following generalization of Theorem \ref{thm:MainThm}:

\begin{thm}
\label{thm:MainThmGeneralized} Choose $\varphi_N$ as in \eqref{ours}, with $\alpha$ Diophantine of type $\kappa$,
and fix $$\eta\in\left( \max\left(\frac{17}{18},1-\frac{2}{9\kappa}\right),1\right].$$
There exists $s=s(\kappa,\eta)>0$ such that for any bounded intervals $A,B_1,B_2\subset\mathbb{R}$
\begin{equation*}
\frac{1}{M^{\eta}}\sum_{M\le N\le M+M^{\eta}}R_{2,N}(A,B_1,B_2;\varphi_N)=|A|\,|B_1|\,| B_2| +O\left(M^{-s}\right)
\end{equation*}
as $M\to\infty$ (the implied constant in the remainder depends on $\alpha,\eta,A,B_1$ and $B_2$).
\end{thm}

By the same approximation arguments of Section  \ref{sec:Proof-of-the} (since the scale of the approximation is larger than $ 1/N $), Theorem \ref{thm:MainThmGeneralized} also holds if we add $ \alpha $ shifts to the summation range in (\ref{generalpair}), i.e., if we define

\begin{alignat}{1}\label{altgeneralpair}
	&R_{2,N}(A,B_1,B_2;\varphi) \\
	 &= \frac{\#\left\{ (i,j)\in {\mathbb Z}^2\mid i\neq j , \varphi(i)-\varphi(j) \in N^{-1} A + \mathbb{Z},  \left(\frac{i-\alpha}{N}, \frac{j-\alpha}{N}\right)\in B_1\times B_2 \right\}}{N}.   \nonumber
\end{alignat}

The proof of Theorem \ref{thm:MainThmGeneralized} goes along similar lines of the proof of Theorem \ref{thm:MainThm}. First, for $ f,h_1,h_2\in \mathcal{S}(\mathbb{R}) $ we define a more general, unequally weighted smooth pair correlation function:

\begin{multline}
\label{eq:SmoothGenPairCorr}R_{2,N}\left(f,g_1,g_2;\varphi\right) \\
=\frac{1}{N}\sum_{m\in\mathbb{Z}}\sum_{i,j\in\mathbb{Z}}h_1\left(\frac{i-\alpha}{N}\right)h_2\left(\frac{j-\alpha}{N}\right)f\left(N\left(\varphi(i)-\varphi(j)+m\right)\right).
\end{multline}

Note that the properties of Section \ref{sec:Jacobi} regarding the absolute square of the Jacobi theta sum hold in fact more generally for $ \Theta_f\left(\tau,\phi;\xi\right) \overline{\Theta_g \left(\tau,\phi;\xi\right)} $ where $ f,g\in\mathcal{S}(\mathbb{R}) $ (see \cite{Marklof}), i.e., $ \Theta_f\left(\tau,\phi;\xi\right) \overline{\Theta_g \left(\tau,\phi;\xi\right)} $ is invariant under the left action of $ \tilde{\Gamma} $; we have 
\[
\int_{\mathbb{T}^{2}}\Theta_f\left(\tau,\phi;\xi\right) \overline{\Theta_g \left(\tau,\phi;\xi\right)}\,\mbox{d}\xi=\left\langle f,g \right\rangle,
\]
and for all $ y\ge 1/2 $, $ A>1 $
\begin{alignat*}{1}
\Theta_f\left(\tau,\phi;\xi\right) \overline{\Theta_g \left(\tau,\phi;\xi\right)}  & =y^{1/2}\sum_{m\in\mathbb{Z}}f_{\phi}\left(y^{1/2}\left(m-\xi_{2}\right)\right)\overline{g_{\phi}\left(y^{1/2}\left(m-\xi_{2}\right)\right)}\\
&+O\left(C_{A,f_\phi,g_\phi} y^{(-A+1)/2}\right)
\end{alignat*}
where  $ C_{A,f_\phi,g_\phi} = \sup_{x\in \mathbb R} \left| \left (1+\left|x\right|\right)^{2A} f_\phi(x) g_\phi(x) \right |$, and the error term is uniform in $x,\phi,\xi$.

Thus, the generalization of Theorem \ref{thm:MainThmSmooth} to the more general smooth pair correlation functions (\ref{eq:SmoothGenPairCorr}) is straightforward.

\begin{thm}
\label{thm:MainThmSmooth22}
Under the assumptions of Theorem \ref{thm:MainThmSmooth}, where in addition $h_1,h_2\in C^{\infty}\left(\mathbb{R}\right)$ are real valued and compactly supported,
\begin{alignat*}{1}
\frac{1}{M^{\eta}}\sum_{N\in\mathbb{Z}} & \psi\left(\frac{N-M}{M^\eta}\right)R_{2,N}\left(\hat{\nu},h_1,h_2;\varphi_N\right) \\
& = \hat{\psi}(0)\left(\nu\left(0\right) \hat{h}_1(0) \hat{h}_2(0) +\hat{\nu}\left(0\right)\left\langle h_1,h_2\right\rangle \right) \\
 & +O\left(L\left\Vert h_1\right\Vert _{C^{n}}\left\Vert h_2\right\Vert _{C^{n}}\left\Vert \nu\right\Vert _{C^{2}} \left\Vert \psi \right\Vert _{C^{8}} M^{-s}\right)
\end{alignat*}
as $M\to\infty$ (the implied constant in the remainder depends on $\alpha,\eta$ and the supports of $ h_1,h_2 $).
\end{thm}

Theorem \ref{thm:MainThmGeneralized} follows by the same approximation arguments of Section \ref{sec:Proof-of-the}.

\section{\label{sec:LongAverages}Long averages}

We show in this appendix that averages of $R_{2,N}$ over long intervals have Poisson
statistics for the sequence $\frac{\beta\left(n-\alpha\right)^{2}}{2N}$ $ (\alpha,\beta \in \mathbb{R}, \beta\ne 0) $
assuming either $\alpha$ or $\beta$ is Diophantine. This follows from the quantitative version of the Oppenheim conjecture for quadratic forms of signature $ (2,2) $ \cite{EskinMargulisMozes,MargulisMohammadi}.

We work with the general pair correlation measure \eqref{generalpair} introduced in Appendix \ref{sec:MoreGeneral}; the same argument will also work for the pair correlation measure \eqref{altgeneralpair}.

\begin{thm}
	\label{thm:LongThm}Fix $\alpha,\beta\in\mathbb{R}$ such that $ \beta\ne0 $ and either
	$\alpha$ or $\beta$ is Diophantine. If $\alpha\notin\frac12\mathbb Z$, then for any bounded intervals $A,B_1,B_2\subset\mathbb{R}$
	\[
	\lim_{M\to\infty}\frac{1}{M}\sum_{N=M}^{2M}R_{2,N}\left(A,B_1,B_2;\varphi_N\right)=|A|\,|B_1|\,|B_2|.
	\]
	If $\alpha\in\frac12 \mathbb Z$, then the above holds provided $B_1,B_2\subset \mathbb R_{\geq 0}$.
\end{thm}

In particular, Theorem \ref{thm:LongThm} holds for a Diophantine
$\alpha$ and $\beta=1$. 

\begin{proof}
	Define the signature $\left(2,2\right)$ quadratic form
	\[
	Q\left(x_{1},x_{2},x_{3},x_{4}\right)=\frac{\beta}{2}\left(x_{1}^{2}-x_{2}^{2}\right)+x_{3}x_{4}.
	\]
	Let $\xi=\left(\alpha,\alpha,0,0\right)$ and define $Q_{\xi}\left(x\right)=Q\left(x-\xi\right).$
	
	Note that
	\begin{multline}
	\frac{1}{M}\sum_{N=M}^{2M}R_{2,N}\left(A,B_1,B_2;\varphi_N\right)= \\ \nonumber
	\frac{1}{M}\sum_{N=M}^{2M}\frac{1}{N}\sum_{m\in\mathbb{Z}}\sum_{\begin{subarray}{c} (\frac{i}{N}, \frac{j}{N})\in B_1\times B_2 \\ i\ne j \end{subarray}}\chi_{A}\left(Q_{\xi}\left(i,j,N,m\right)\right).
	\end{multline}
	Let 
	\[
	f\left(x_{1},x_{2},x_{3}\right)=\frac{1}{x_3}\chi_{B_1}\left(\frac{x_{1}}{x_{3}}\right)\chi_{B_2}\left(\frac{x_{2}}{x_{3}}\right)\chi_{\left[1,2\right]}\left(x_{3}\right).
	\]
	Then
	\[
	\frac{1}{M}\sum_{N=M}^{2M}R_{2,N}\left(A,B_1,B_2;\varphi_N\right)=\frac{1}{M^{2}}\sum_{\begin{subarray}{c}
		x\in\mathbb{Z}^{4}\\
		x_{1}\ne x_{2}
		\end{subarray}}f\left(\frac{x_{1}}{M},\frac{x_{2}}{M},\frac{x_{3}}{M}\right)\chi_{A}\left(Q_{\xi}\left(x\right)\right).
	\]
	Since either $\alpha$ or $\beta$ is Diophantine, the form $Q_{\xi}$
	is Diophantine (recall \cite[Definition 1.8]{MargulisMohammadi}). Moreover, let $x\in\mathbb{Z}^{4}$
	with $\left(\frac{x_1}{x_3}, \frac{x_2}{x_3}\right)\in B_1\times B_2$, $M\le x_{3}\le2M$ and $Q_{\xi}\left(x\right) \in A$.
	If $\alpha\notin\frac12\mathbb Z$, then $x$ belongs to an exceptional subspace of $Q_{\xi}$ if and
	only if $x_{1}=x_{2}$. If $\alpha\in\frac12\mathbb Z$, then  $x$ belongs to an exceptional subspace of $Q_{\xi}$ if and
	only if $x_{1}=x_{2}$ or $ x_1+x_2= 2\alpha $; since in that case we assume that $B_1,B_2 \subset \mathbb R_{\geq 0}$, the condition $ x_1+x_2= 2\alpha $ can  only occur for a bounded number of $ x_1,x_2 $ whose contribution to the sum is $ o(1) $.
	
	Since for any $ 0 \le c_1 < c_2 $ and $ C>0 $ the set \[ \left\{x\in \mathbb{R}^4: \left(\frac{x_1}{x_3}, \frac{x_2}{x_3}\right)\in B_1\times B_2, c_1 \le x_3 \le c_2, |x_4|\le C \right\} \] can be written as a difference of two star-shaped regions, we deduce from \cite[Theorem 1.9]{MargulisMohammadi} using a standard approximation argument that
	\begin{multline}
		\frac{1}{M}\sum_{N=M}^{2M}R_{2,N}\left(A,B_1,B_2;\varphi_N\right) \sim \frac{1}{M^{2}}\int_{\mathbb{R}^{4}}f\left(\frac{x_{1}}{M},\frac{x_{2}}{M},\frac{x_{3}}{M}\right)\chi_{A}\left(Q_{\xi}\left(x\right)\right)\,\mbox{d}x \\
		=\frac{1}{M} \int_M^{2M} \, \frac{ \mbox{d}x_3 }{x_3} \int_{x_3 B_1} \,\mbox{d}x_1 \int_{x_3 B_2}\,\mbox{d}x_2 \int_{Q_{\xi}\left(x\right) \in A} \,\mbox{d}x_4 
		= |A||B_1||B_2|. \nonumber
	\end{multline}
\end{proof}

\section{\label{sec:Equidistribution-on-Heisenberg}Conjecture \ref{conj:EquiConj} implies \texorpdfstring{$R_{2,N}$}{R2,N} is Poisson}

We show in this appendix that, assuming Conjecture \ref{conj:EquiConj}, the pair correlation measure $R_{2,N}$ converges weakly to Lebesgue measure (without averaging on $N$).

\subsection{Proof of Proposition \ref{prop:ConjecturedProp}}
\begin{proof}
By an approximation argument similar to that of \S\ref{sec:Proof-of-the},
it is enough to show that there exist $l\in\mathbb{N}, \kappa_0\ge 2 $, such that for any Diophantine $ \alpha $ of type $ \kappa\le\kappa_0 $, there exist $ s=s(\kappa)>0, n=n(\kappa)\in \mathbb{N} $ such that for any $\nu\in C^{\infty}\left(\mathbb{R}\right)$ supported on $ \left[-L,L\right] $ $\left(L\ge1\right)$,
$h\in C^{\infty}\left(\mathbb{R}\right)$ real valued and supported on $ \left[-1/2,3/2\right] $,
\begin{equation}
Q_{N}\left(\nu,h\right)=\hat{\nu}\left(0\right)\left\Vert h\right\Vert _{2}^{2}+O\left(L\left\Vert h\right\Vert _{C^n}^2 \left\Vert \nu\right\Vert _{C^{l}}N^{-s}\right)\label{eq:PointwiseProp}
\end{equation}
as $ N\to\infty $.

To prove (\ref{eq:PointwiseProp}), recall that by Lemma \ref{lem:cuspEst}, for $ 0<\delta<\frac{\kappa-1}{\kappa} $ and for $ 0<\epsilon<\frac{\delta}{\kappa-1} $, there exists $ n=n(\kappa,\delta,\epsilon)\in\mathbb{N} $ such that for any $\nu\in C^{\infty}\left(\mathbb{R}\right)$ supported on $ \left[-L,L\right] $ $\left(L\ge1\right)$,
$h\in C^{\infty}\left(\mathbb{R}\right)$ real valued and supported on $ \left[-1/2,3/2\right] $,
\begin{alignat*}{1}
	Q_{N}\left(\nu,h\right) & =\frac{1}{N}\sum_{\begin{subarray}{c}
			\sigma\mid N\\
			\sigma\le N^{\delta}
	\end{subarray}}\sum_{\left(k,N/\sigma\right)=1}\nu\left(\frac{k}{N/\sigma}\right)\left|\Theta_{h}\left(\frac{k}{N/\sigma}+\frac{i}{N^{2}},0,\begin{pmatrix}
		0\\
		\alpha
	\end{pmatrix}\right)\right|^{2}\\
	& +O\left(L\left\Vert h\right\Vert _{C^n}^2\left\Vert \nu\right\Vert _{\infty}N^{-\frac{\delta}{\kappa-1}+\epsilon}\right)
\end{alignat*}
as $ N\to \infty $.

In addition,
\begin{alignat}{1}
 & \frac{1}{N}\sum_{\begin{subarray}{c}
\sigma\mid N\\
\sigma\le N^\delta
\end{subarray}}\sum_{\left(k,N/\sigma\right)=1}\nu\left(\frac{k}{N/\sigma}\right)\left|\Theta_{h}\left(\frac{k}{N/\sigma}+\frac{i}{N^{2}},0,\begin{pmatrix}
0\\
\alpha
\end{pmatrix}\right)\right|^{2}\nonumber \\
 & =\frac{1}{N}\sum_{\begin{subarray}{c}
 		\sigma\mid N\\
 		\sigma\le N^\delta
 \end{subarray}}\sum_{\left(k,N/\sigma\right)=1}\nu\left(\frac{k}{N/\sigma}\right)f_{\sigma^{2},\frac{\pi}{2}}\left(\frac{a}{N/\sigma}, \begin{pmatrix}
 	a & b\\
 	N/\sigma & -k
 \end{pmatrix} \begin{pmatrix}
 	\frac 1 2\\
 	\alpha+ \frac 1 2\end{pmatrix}\right)\label{eq:2PointNoAverage}
\end{alignat}
where $\left(a,b\right)$ are (any) integers that solve the equation
$ak+bN/\sigma=-1$. Assuming conjecture \ref{conj:EquiConj},  there exist $k,l\in\mathbb{N}, \kappa_0\ge 2 $, such that for any Diophantine  $ \alpha $ of type $ \kappa\le\kappa_0 $, there exists $ t=t(\kappa)>0  $ such that the
r.h.s. of (\ref{eq:2PointNoAverage}) is equal to
\begin{gather}
	\frac{1}{N}\sum_{\begin{subarray}{c}
			\sigma\mid N\\
			\sigma\le N^{\delta}
	\end{subarray}}\phi\left(\frac{N}{\sigma}\right)\left(\hat{\nu}\left(0\right)\left\Vert h\right\Vert _{2}^{2}+O\left(L \left\Vert \nu\right\Vert _{C^{l}}\left\Vert f_{\sigma^{2},\frac{\pi}{2}}\right\Vert _{C^{k}}\left(N/\sigma\right)^{-t}\right)\right).\label{eq:equidistLimit}
\end{gather}
Since 
\[
\frac{1}{N}\sum_{\begin{subarray}{c}
	\sigma\mid N\\
	\sigma>N^{\delta}
	\end{subarray}}\phi\left(\frac{N}{\sigma}\right)=\frac{1}{N}\sum_{\begin{subarray}{c}
	\sigma\mid N\\
	\sigma<N^{1-\delta}
	\end{subarray}}\phi\left(\sigma\right)\ll N^{-\delta+\epsilon},
\]
we have
\[
\frac{1}{N}\sum_{\begin{subarray}{c}
	\sigma\mid N\\
	\sigma\le N^{\delta}
	\end{subarray}}\phi\left(\frac{N}{\sigma}\right)=1+O\left(N^{-\delta+\epsilon}\right).
\]
By (\ref{eq:Theta_est}), there exists $ n=n(k) $ such that that $\left\Vert f_{\sigma^{2},\frac{\pi}{2}}\right\Vert _{C^{k}}\ll \left\Vert h\right\Vert _{C^{n}}^2\sigma^{k+1}.$
Thus we conclude that (\ref{eq:equidistLimit}) is equal to
\[
\hat{\nu}\left(0\right)\left\Vert h\right\Vert _{2}^{2}+O\left(L\left\Vert h\right\Vert _{C^n}^2 \left\Vert \nu\right\Vert _{C^{l}}\left(N^{-\delta+\epsilon}+N^{-t+\delta\left(k+t\right)+\epsilon}\right)\right)
\]
and taking $\delta>0$ sufficiently small, Proposition \ref{prop:ConjecturedProp}
follows.\end{proof}


\begin{thebibliography}{1}
	
\bibitem{BerryTabor}M. V. Berry and M. Tabor, \emph{Level clustering in the regular spectrum}.  Proc. Roy. Soc. A. \textbf{356} (1977), 375\textendash 394.

\bibitem{Bleher}P. M. Bleher, \emph{The energy level spacing for two harmonic oscillators with generic ratio of frequencies}. J. Statist. Phys. \textbf{63} (1991), no. 1-2, 261\textendash 283.
	
\bibitem{BocaZaharescu}F. P. Boca and A. Zaharescu, \emph{Pair correlation of values of rational functions (mod $ p $)}. Duke Math. J. \textbf{105} (2000),
no. 2, 267\textendash 307.

\bibitem{DeBievreEspostiGiachetti}S. De Bièvre, M. Degli Esposti and R. Giachetti, \emph{Quantization of a class of piecewise affine transformations on the torus}. Comm. Math. Phys. \textbf{176} (1996),
no. 1, 73\textendash 94.

\bibitem{MarklofCellarosi}F. Cellarosi and J. Marklof,  \emph{Quadratic Weyl sums, automorphic functions, and invariance principles}. Proceedings of the London Mathematical Society. \textbf{113} (2016), no. 6, 775\textendash828.


\bibitem{EskinMargulisMozes}A. Eskin, G. Margulis and S. Mozes,  \emph{Quadratic forms of signature $(2,2)$ and eigenvalue spacings on rectangular $2$-tori}.  Ann. of Math. (2) \textbf{161} (2005),
no. 2, 679\textendash 725.


\bibitem{Greenman}C. D. Greenman, \emph{The generic spacing distribution of the two-dimensional harmonic oscillator}. J. Phys. A. \textbf{29} (1996), no. 14, 4065\textendash 4081.

\bibitem{HeathBrown}D. R. Heath-Brown, \emph{Pair correlation for fractional parts of $ \alpha n^2 $}. Math. Proc. Cambridge Philos. Soc.  \textbf{148} (2010),
no. 3, 385\textendash 407.

\bibitem{KurlbergRudnick}P. Kurlberg and Z. Rudnick, \emph{The distribution of spacings between quadratic residues}. Duke Math. J. \textbf{100} (1999),
no. 2, 211\textendash 242.

\bibitem{MargulisMohammadi}G. Margulis and A. Mohammadi,  \emph{Quantitative version of the Oppenheim conjecture for inhomogeneous quadratic forms}. Duke Math. J. \textbf{158} (2011), no. 1, 121\textendash 160.

\bibitem{MarklofECM}J. Marklof, \emph{The Berry-Tabor conjecture}. European Congress of Mathematics, Vol. II (Barcelona, 2000), Progr. Math. \textbf{202} (2001), Birkhäuser, Basel, 421\textendash 427.

\bibitem{Marklof}J. Marklof, \emph{Pair correlation densities of
inhomogeneous quadratic forms}. Ann. of Math. (2) \textbf{158} (2003),
no. 2, 419\textendash 471.



\bibitem{MarklofStrombergsson}J. Marklof and A. Strömbergsson, \emph{Equidistribution of Kronecker sequences along closed horocycles}. Geom. Funct. Anal. \textbf{13} (2003),
no. 6, 1239\textendash 1280.

\bibitem{PandeyBohigasGiannoni}A. Pandey, O. Bohigas and M. J. Giannoni, \emph{Level repulsion in the spectrum of two-dimensional harmonic oscillators}. J. Phys. A. \textbf{22} (1989), no. 18, 4083\textendash 4088.

\bibitem{Pellegrinotti}A. Pellegrinotti, \emph{Evidence for the poisson distribution for quasi-energies in the quantum kicked-rotator model}. J. Stat. Phys. \textbf{53} (1988),  no. 5-6, 1327\textendash 1336.

\bibitem{RudnickSarnak}Z. Rudnick and P. Sarnak, \emph{The pair correlation function of fractional parts of polynomials}. Comm. Math. Phys. \textbf{194} (1998), no. 1, 61\textendash 70.

\bibitem{RudnickSarnakZaharescu}Z. Rudnick, P. Sarnak and A. Zaharescu, \emph{The distribution of spacings between the fractional parts of $ n^2 \alpha $}. Invent. Math. \textbf{145} (2001), no. 1, 37\textendash 57.

\bibitem{Sinai}Ya. G. Sinai, \emph{The absence of the Poisson distribution for spacings between quasi-energies in the quantum kicked-rotator model}. Phys. D. \textbf{33} (1988),  no. 1-3, 314\textendash 316.

\bibitem{Slater}N. Slater, \emph{Gaps and steps for the sequence $ n\theta \mod 1 $}. Proc. Cambridge Philos. Soc. \textbf{63} (1967), 1115\textendash 1123.

\bibitem{Strombergsson}A. Strömbergsson, \emph{An effective Ratner
	equidistribution result for $\mbox{SL}\left(2,\mathbb{R}\right)\ltimes\mathbb{R}^{2}$}.
Duke Math. J. \textbf{164} (2015), no. 5, 843\textendash 902.

\bibitem{Truelsen}J. L. Truelsen, \emph{Divisor problems and the pair correlation for the fractional parts of $ n^2 \alpha $}.  Int. Math. Res. Not. (2010), no. 16, 3144\textendash 3183.

\bibitem{Weyl}H. Weyl, \emph{Über die Gleichverteilung von Zahlen mod. Eins}. Math. Ann. \textbf{77} (1916), no. 3, 313\textendash 352.

\bibitem{Zelditch}S. Zelditch, \emph{Level spacings for integrable quantum maps in genus zero}. Comm. Math. Phys. \textbf{196} (1998), no. 2, 289\textendash 318. Addendum: ``Level spacings for integrable quantum maps in genus zero'', ibid., 319-329.

\bibitem{ZelditchZworski}S. Zelditch and M. Zworski, \emph{Spacing between phase shifts in a simple scattering problem}. Comm. Math. Phys. \textbf{204} (1999), no. 3, 709\textendash 729.

\end{thebibliography}
\end{document}